\documentclass{aims}
\usepackage{amsmath}
  \usepackage{paralist}
 \usepackage[colorlinks=true]{hyperref}
\hypersetup{urlcolor=blue, citecolor=red}

\def\sp{\text{sp }}
\def\qandq{\quad\text{and}\quad}
\def\where{\quad\text{where}}
\def\del{\overline{\nabla}}
\def\div{\nabla\cdot}
\def\sgn{\text{sign }}
\def\Cm{\Bbb{C}}
\def\Imm{\Bbb{I}}
\def\Mm{\Bbb{M}}
\def\Rm{\Bbb{R}}
\def\Sm{\Bbb{S}}
\def\Um{\Bbb{U}}
\def\bfe{\bold{e}}
\def\bff{\bold{f}}
\def\bfg{\bold{g}}
\def\bfS{\bold{S}}
\def\bfk{\bold{k}}
\def\bft{\bold{t}}
\def\bfu{\bold{u}}
\def\bfv{\bold{v}}

\def\bfI{\bold{I}}
\def\bfP{\bold{P}}
\def\bmrho{\boldsymbol{\rho}}
\def\bmphi{\boldsymbol{\varphi}}
\def\bmnu{\boldsymbol{\nu}}

\def\C{\mathcal{C}}
\def\H{\mathcal{H}}
\def\L{\mathcal{L}}
\def\J{\mathcal{J}}
\def\O{\mathcal{O}}
\def\R{\mathcal{R}}
     
\def\currentvolume{X}
 \def\currentissue{X}
  \def\currentyear{200X}
   \def\currentmonth{XX}
    \def\ppages{X--XX}

\newtheorem{theorem}{Theorem}[section]

\newtheorem{lemma}[theorem]{Lemma}
\newtheorem{proposition}{Proposition}

\theoremstyle{definition}

\newtheorem{remark}{Remark}

\newcommand{\ep}{\varepsilon}

\title[Inverse diffusion from internal data]
      {Inverse diffusion problems with redundant internal information}

\author[Fran\c cois Monard and Guillaume Bal]{}

\subjclass{Primary: 35R30 ; Secondary: 35J45, 53B21.}
 \keywords{Inverse conductivity, Calder\'on's problem, hybrid methods, power density measurements, strongly coupled elliptic systems, differential geometry.}

 \email{fm2234@columbia.edu}
 \email{gb2030@columbia.edu}

\thanks{The authors are supported by NSF under grant DMS-0804696}

\begin{document}
\maketitle

\centerline{\scshape Fran\c cois Monard }
\medskip
{\footnotesize
 \centerline{Department of Applied Physics and Applied Mathematics, Columbia University}
   \centerline{ New York NY, 10027, USA}
} 

\medskip

\centerline{\scshape Guillaume Bal}
\medskip
{\footnotesize
 \centerline{ Department of Applied Physics and Applied Mathematics, Columbia University}
   \centerline{New York NY, 10027, USA}
}

\bigskip

 \centerline{(Communicated by the associate editor name)}

\begin{abstract}
   This paper concerns the reconstruction of a scalar diffusion coefficient $\sigma(x)$ from redundant functionals of the form $H_i(x)=\sigma^{2\alpha}(x)|\nabla u_i|^2(x)$ where $\alpha\in\Rm$ and  $u_i$ is a solution of the elliptic problem $\nabla\cdot \sigma \nabla u_i=0$ for $1\leq i\leq I$. The case $\alpha=\frac12$ is used to model measurements obtained from modulating a domain of interest by ultrasound and finds applications in ultrasound modulated electrical impedance tomography (UMEIT), ultrasound modulated optical tomography (UMOT) as well as impedance acoustic computerized tomography (ImpACT). The case $\alpha=1$ finds applications in Magnetic Resonance Electrical Impedance Tomography (MREIT). 
   
   We present two explicit reconstruction procedures of $\sigma$ for appropriate choices of $I$ and of traces of $u_i$ at the boundary of a domain of interest. The first procedure involves the solution of an over-determined system of ordinary differential equations and generalizes to the multi-dimensional case and to (almost) arbitrary values of $\alpha$ the results obtained in two and three dimensions in \cite{CFGK} and \cite{BBMT}, respectively, in the case $\alpha=\frac12$. The second procedure consists of solving a system of linear elliptic equations, which we can prove admits a unique solution in specific situations. 
\end{abstract}

\section{Introduction} \label{sec:intro}

Medical imaging modalities aim to combine high resolution with high contrast between healthy and unhealthy tissues. Optical Tomography and Electrical Impedance Tomography display such high contrasts but often suffer from poor resolution. Ultrasound Tomography and Magnetic Resonance Imaging are high resolution modalities that sometimes suffer from low contrast. The ultrasound modulation of electrical or optical properties of tissues and the combination of simultaneous electrical and magnetic resonance measurements both offer the possibility to combine high resolution with high contrast. For the acquisition of ultrasound-modulated measurements, we refer the reader to, e.g., \cite{ABCTF,BBMT,CFGK,GS-SIAP-09,KK} for works in the mathematical literature. For the acquisition of internal information on electrical conductivities by magnetic resonance imaging, we refer the reader to, e.g., \cite{KKSY-SIMA-02,NTT-IP-07,NTT-IP-09,NTT-Rev-11}.

Mathematically, we aim to reconstruct a scalar diffusion coefficient $\sigma$ in an elliptic equation from knowledge of internal information of the form $H_{ij}(x)=\sigma^{2\alpha} (x)\nabla u_i(x)\cdot\nabla u_j(x)$ for $\alpha\in\Rm$ and $1\leq i,j\leq m$, where $u_i$ and $u_j$ are solutions of the elliptic problem with different boundary conditions; see \eqref{eq:conductivity} and \eqref{eq:Hij} below. Coupling impedance (or diffusion) with acoustic waves or magnetic resonance correspond to the cases $\alpha=\frac12$ and $\alpha=1$, respectively. Such information can be obtained from functionals of the form $\sigma^{2\alpha}(x) |\nabla u_i|^2(x)$ by standard polarization (expressions of the form $4ab=(a+b)^2-(a-b)^2$).

This problem was first solved in the two dimensional setting in \cite{CFGK} in the case $m=2$ and $\alpha=\frac12$. The three dimensional setting was addressed in \cite{BBMT} with $m=4$ and $\alpha=\frac12$. In these papers, the elliptic equation is recast as a system of equations for quantities of the form $S_i=\sigma^\alpha \nabla u_i$ using the elliptic equation and the fact that $\nabla u_i$ is curl free. This strategy allows one to eliminate $\sigma$ from the system of equations and solve for the vectors $S_i$. The stable reconstruction of $\sigma$ is then straightforward. The case $\alpha=\frac12$ in the setting of non-redundant measurements, i.e., with $m=1$ and measurements of the form $H=\sigma|\nabla u|^2$ is considered in \cite{B-UMEIT-11}. It is shown in that paper that the stable reconstruction of $\sigma$ may not be possible from such non-redundant measurements. This justifies the analysis of redundant measurements.

The objectives of this paper are twofold. We first generalize the reconstruction of $\sigma$ to the case of arbitrary space dimension $n$ and almost arbitrary $\alpha\in\Rm$. Assuming that the vectors $S_i$ form a frame, we obtain a system of equations for the vectors $S_i$ that involves $H_{ij}=S_i\cdot S_j$  but no longer $\sigma$. The resulting system of equations may be seen as an overdetermined nonlinear system of equations. By appropriately choosing the boundary conditions used to construct the internal functionals $H_{ij}(x)$, we obtain a global uniqueness and stability result for the reconstruction of the scalar quantity $\sigma(x)$. Although several portions of the algorithm generalize to the reconstruction of anisotropic diffusion tensors, we restrict ourselves to the scalar case in this paper.  We also describe and investigate the compatibility conditions associated with such a redundant system. 

The second objective of the paper is to present a system of elliptic equations for the solutions $u_i$ with constitutive parameters that depend on the measurements $H_{ij}$ but not on the unknown diffusion coefficient $\sigma$. We show that the system is uniquely solvable when a Fredholm alternative holds. We obtain existence and uniqueness results for the proposed system for all but a discrete number of values of the dimension $n$ and the coefficient $\alpha\in\Rm$. 

Both algorithms require boundary conditions for the elliptic solutions that ensure that $n$ of the vectors $S_i$ form a frame in $\Rm^n$ at each point of the domain of interest. Whereas such a condition is easy to meet in two space dimensions, in dimensions three and higher, the only available technique that guarantees such an independence is based on using complex geometrical optics (CGO) solutions. We generalize here the CGO construction of \cite{BBMT} to the multi-dimensional setting and for almost all values of $\alpha$.

The inverse diffusion problems with internal functionals considered here are examples of hybrid inverse problem where two imaging modalities are combined to provide both high resolution and high contrast. For recent works on the mathematics of hybrid inverse problems and their many applications in medical imaging, we refer the reader to the articles in the book \cite{S-Sp-2011} and to the recent review paper \cite{B}.

The rest of the paper is structured as follows. Section \ref{sec:results} presents the main results of the paper on the stable reconstruction of $\sigma$ from available internal functionals. The elimination of $\sigma$ from the system of equations for the vectors $S_i$ and the corresponding differential calculus is explained in section \ref{sec:datatoF}. The redundant system of equations for the vectors $S_i$ is addressed in section \ref{sec:ODE}  while the system of linear equations for the solutions $u_i$ is given in section \ref{sec:elliptic}. Finally, section \ref{sec:comp} presents further reconstruction algorithms in the two dimensional case and analyzes the compatibility conditions satisfied by the redundant data and their potential use.

\section{Statement of the main results} \label{sec:results}
Let $X$ be an open convex bounded domain of $\Rm^n$ with $n\ge 2$. 
In the following, we address the reconstruction of the scalar conductivity (or diffusion) coefficient $\sigma$ in the equation 
\begin{align}
    \nabla\cdot(\sigma\nabla u_i) = 0 \quad (X), \quad u_i|_{\partial X} = g_i, \quad 1\le i\le m,
    \label{eq:conductivity}
\end{align}
where $m\ge n$, from knowledge of the interior functionals
\begin{align}
    H_{ij} (x) = \sigma(x)^{2\alpha} \nabla u_i (x) \cdot\nabla u_j (x), \qquad 1\le j\le i\le m,
    \label{eq:Hij}
\end{align}
where $\alpha\in\Rm$ is fixed and such that $(n-2)\alpha+1\not=0$. The derivation of the internal functionals \eqref{eq:Hij} in the case $\alpha = \frac{1}{2}$ is detailed in \cite{BBMT,KK} as examples of synthesized focusing, in \cite{ABCTF} in a setup of temporal, physical, focusing, and in \cite{GS-SIAP-09} by considering thermoelastic effects. The case $\alpha=1$ with $m=1$ related to MREIT and CDII (Current Density Impedance Imaging) is addressed in \cite{KKSY-SIMA-02,NTT-IP-07,NTT-IP-09,NTT-Rev-11}.

Following a similar approach to \cite{BBMT,CFGK}, we first perform the change of unknown functions $S_i = \sigma^\alpha \nabla u_i$ for every $i$ and define
\begin{equation}
  \label{eq:F}F(x):= \nabla\log\sigma(x).
\end{equation}
We also equip $X\subset \Rm^n$ with its Euclidean metric $g_{ij} = \delta_{ij}$ in the canonical basis $(\bfe_1, \dots, \bfe_n)$. For a given vector field $V = V^i \bfe_i$ defined on $X$, we define the corresponding one-form $V^\flat:= V^i dx^i$ (i.e., by means of the \emph{flat operator}). With this notation, we obtain that the vector fields $S_j$ satisfy the system of equations
\begin{align}
    \div S_j &= - (1-\alpha) F\cdot S_j, \label{eq:divSj} \\
    d S_j^\flat &= \alpha F^\flat\wedge S_j^\flat, \qquad 1\le j\le m, \label{eq:dSj}
\end{align}
where $\wedge$ and $d$ denote the usual exterior product and exterior derivative, respectively. The first equation stems directly from \eqref{eq:conductivity} whereas the second one states that the one-form $\sigma^{-\alpha} S_j^\flat = du_j$ is exact, therefore closed, and hence $d(\sigma^{-\alpha} S_j^\flat)=0$. When $n=2,3$, equation \eqref{eq:dSj} is recast as: 
\begin{align*}
    n=2:\quad \nabla^\perp\cdot S_j - \alpha JF\cdot S_j = 0, \qquad  n=3:\quad \text{curl } S_j - \alpha F\times S_j = 0,
\end{align*}
where in dimension $n=2$, we define $J := \left[ \begin{smallmatrix} 0 & -1 \\ 1 & 0 \end{smallmatrix} \right]$ and $\nabla^\perp := J\nabla$, and in dimension $n=3$, $\times$ denotes the standard cross-product. The available information becomes $H_{ij}(x) = S_i(x)\cdot S_j(x)$. 

A crucial hypothesis for our reconstruction procedure is that the $m$ gradients have maximal rank in $\Rm^n$ at every point $x\in X$. This hypothesis can be formalized by the somewhat stronger statement: there exists a finite open covering $\O = \{\Omega_k\}_{1\le k\le N}$ of $X$ (i.e. $X\subset\cup_{k=1}^N \Omega_k$), an indexing function $\tau:[1,N]\ni i\mapsto \tau(i) = (\tau(i)_1,\dots,\tau(i)_n)\in [1,m]^n$ and a positive constant $c_0$ such that 
\begin{align}
    \min_{1\le i\le N} \inf_{x\in\Omega_i} \det(S_{\tau(i)_1} (x), \dots, S_{\tau(i)_n} (x)) \ge c_0 >0.
    \label{eq:positivityS}
\end{align}
This assumption is equivalent to imposing the following condition on the data
\begin{align}
    \min_{1\le i\le N} \inf_{x\in\Omega_i} \det H^{\tau(i)}(x) \ge c_0^2 >0,
    \label{eq:positivity}
\end{align}
where $H^{\tau(i)}$ stands for the $n\times n$ matrix of elements $H^{\tau(i)}_{kl} = S_{\tau(i)_k}\cdot S_{\tau(i)_l}$.
While one can always find illuminations such that \eqref{eq:positivityS} holds in two dimensions with $m=n=2$ and $\O = \{X\}$ (the most preferrable case) by virtue of \cite[Theorem 4]{AN}, higher dimensions can be dealt with using complex geometrical optics solutions provided that $\sigma$ has enough regularity, as the following lemma shows.
\begin{lemma}\label{lem:CGO}
    Let $n\ge 3$ and $\sigma\in H^{\frac{n}{2}+3+\ep}(X)$ for some $\ep>0$ be bounded from below by a positive constant. Then
    \begin{itemize}
	\item[(i)] for $n$ even, there exists a non-empty open set $G$ of illuminations $\{g_1,\dots,g_n\}$ such that for any $\bfg\in G$, the condition \eqref{eq:positivity} holds with $\O=\{X\}$ for some constant $c_0 >0$. 
	\item[(ii)] For $n$ odd, there exists a non-empty set $G$ of illuminations $\{g_1,..,g_{n+1}\}$ such that for any $\bfg\in G$ there exists an open cover of $X$ of the form $\left\{ \Omega_{2i-1}, \Omega_{2i} \right\}_{1\le i\le N}$ and a constant $c_0>0$ such that 
    \begin{align}
	\inf_{x\in\Omega_{2i-1}} \det(S_1,\dots,S_{n-1}, \ep_i S_n) \ge c_0 \qandq \inf_{x\in\Omega_{2i}} \det(S_1,\dots,S_{n-1}, \tilde\ep_i S_{n+1}) \ge c_0,
	\label{eq:positivitynodd}
    \end{align}
    for $1\le i\le N$ and with $\ep_i,\tilde\ep_i = \pm 1$. 
    \end{itemize}
\end{lemma}
The first step toward an inversion is to express the source term $F$ in terms of a local frame:
\begin{lemma} \label{lem:datatoF}
    Let $\Omega\subset X$ open where, up to renumbering solutions, we have 
    \begin{align*}
	\inf_{x\in \Omega} \det(S_1(x),\dots,S_n(x)) \ge c_0 >0.
    \end{align*}
    Then at every point $x\in\Omega$ and denoting $H(x):= \{S_i(x)\cdot S_j(x)\}_{1\le i,j\le n}$, $D(x) = \sqrt{\det H(x)}$, the vector field $F(x) = \nabla\log\sigma(x)$ is given by the following formulas
    \begin{align}
	\begin{split}
	    F &= \frac{c_F}{D}\sum_{i,j=1}^n (\nabla (D H^{ij}) \cdot S_i) S_j = c_F \Big(\nabla\log D + \sum_{i,j=1}^n (\nabla H^{ij} \cdot S_i) S_j \Big),  \\
	    c_F &:= ((n-2)\alpha+1)^{-1}.
	\end{split}
	\label{eq:datatoF}
    \end{align}
    where $H^{ij}$ denotes the element $(i,j)$ of the matrix $H^{-1}$. 
\end{lemma}
Formula \eqref{eq:datatoF} was first proved in \cite{BBMT} in the two- and three-dimensional cases with $\alpha=\frac{1}{2}$ and is here proved for general $n$ and $\alpha\in\Rm$ such that $(n-2)\alpha+1\ne 0$. This formula gives us a way to reconstruct $F$ locally from $n$ linearly independent solutions. Assuming condition \eqref{eq:positivity}, one can then reconstruct $F$ globally over $X$. 

From lemma \ref{lem:datatoF}, one can follow two directions to reconstruct the conductivity, which we now describe in more detail in the next two paragraphs. 
\paragraph{\underline{The ODE-based reconstruction procedure}}
The first approach consists in plugging equation \eqref{eq:datatoF} back into the system \eqref{eq:divSj}-\eqref{eq:dSj} and obtain a closed system for the vectors $S_j$. We then show that the resulting system leads to a gradient system, which can then be solved for the vectors $S_j$ by ODE integration. Once the vectors $S_j$ are reconstructed, one recovers $\sigma$ from the knowledge of its value at a given point and the fact that $\nabla\log\sigma$ is now known by equation \eqref{eq:datatoF}. This approach is a generalization of the results of \cite{BBMT} to higher-dimensional settings and general $\alpha\in \Rm$ such that $(n-2)\alpha+1\ne 0$, and leads to well-posed reconstructions as stated in the following:
\begin{theorem}[Global uniqueness and stability, ODE-based reconstruction procedure] \label{thm:ODEstab}
    Let $X\subset\Rm^n, n\ge 2$ be an open convex bounded domain, and let two sets of $m\ge n$ solutions of \eqref{eq:conductivity} generate measurements $(H,H')$ whose components belong to $W^{1,\infty}(X)$, and who jointly satisfy condition \eqref{eq:positivity} with the same triple $(\O,\tau,c_0)$. Let also $x_0\in\overline{\Omega_{i_0}}\subset \overline X$ and $\sigma(x_0), \sigma'(x_0)$ and $\{ S_{\tau(i_0)_i}(x_0), S'_{\tau(i_0)_i}(x_0) \}_{1\le i\le n}$ be given. Let $\sigma$ and $\sigma'$ be the conductivities corresponding to the measurements $H$ and $H'$, respectively. Then we have the stability estimate:
    \begin{align}
	\|\log\sigma-\log\sigma'\|_{W^{1,\infty}(X)} \le C \left(\ep_0 + \|H-H'\|_{W^{1,\infty}(X)} \right),
	\label{eq:ODEstab}
    \end{align}
    where $\ep_0$ is the error committed at the point $x_0$:
    \begin{align*}
	\ep_0 := |\log \sigma(x_0) - \log\sigma'(x_0)| + \sum_{i=1}^n \|S_{\tau(i_0)_i}(x_0)- S'_{\tau(i_0)_i}(x_0)\|.
    \end{align*}
\end{theorem}

The solution for the vectors $S_i$ and then for $\log\sigma$ requires the solution of full gradient equations of the form $\nabla u = f(u)$, where $u$ stands for either unknown. These overdetermined PDEs require compatibility conditions on $f$ if we wish to ensure that their solution does not depend on the path of integration. Theorem \ref{thm:ODEstab} shows that the reconstruction is unique and stable with respect to the data once a fixed family of integration curves is chosen. The compatibility conditions addressed in section \ref{sec:comp} are shown to depend quadratically on the unknown frame. It is therefore difficult to enforce them while solving for the frame $S$. Nonetheless, depending on the value of $\alpha$, they may lead to algebraic (i.e. pointwise) reconstructions of all or part of the unknown frame, and may also provide further conditions on the data $H_{ij}$. Such analyses are carried out in section \ref{sec:comp}.  

\begin{remark} \label{rem:Rframe}
    Solving a system of equations for the unknown $S_i$ may not be efficient numerically. Let $S$ be the matrix whose columns are the $n$ linearly independent vectors $S_j$ at a given $x$. Then $S^TS=H$ is known.  By the Gram-Schmidt (GS) orthonormalization procedure or by setting $R=SH^{-\frac{1}{2}}$, we can write an equation for an oriented orthonormal frame $R$; see section \ref{sec:Rframe} below. This approach requires that we reconstruct $n(n-1)/2 = \dim SO_n(\Rm)$ scalar functions instead of the $n\times m$ components of the vector fields $\{S_j\}$. The only additional constraint is that the transition matrix from $S$ to $R$ satisfies a certain stability property with respect to the data $H$, see Section \ref{sec:Rframe} for details.
\end{remark}

\begin{remark}\label{rem:alphazero}
   The case $\alpha=0$, corresponding to information of the form $H_{ij}(x)=\nabla u_i(x)\cdot\nabla u_j(x)=S_i(x)\cdot S_j(x)$ simplifies in the sense that the elimination of $F$ is not necessary. Indeed, we show in the next section that knowledge of $S_i\cdot S_j$ and the constraints $dS_j^\flat=0$ for $1\leq j\leq m$ uniquely determine the vectors $S_j$ provided they are known at one point $x_0$. Once $\nabla u_i$ is known, the reconstruction of $\sigma$ may proceed from using \eqref{eq:datatoF}. Note that, alternatively, the equation \eqref{eq:conductivity} may be seen as a transport equation for $\sigma^{1-\alpha}$ when $\alpha\not=1$ once the vector field $\sigma^\alpha\nabla u$ is known. The stability properties of such a reconstruction are established in \cite{BR-IP-11,BU-IP-10}.
\end{remark}

\paragraph{\underline{The elliptic-based reconstruction procedure}}
The second approach is novel and consists in injecting equation \eqref{eq:datatoF} back into the initial conductivity equations and obtain a \emph{strongly coupled elliptic system} of the form
\begin{align}
    \Delta u_i + c_F W_{ij}\cdot \nabla u_j = 0, \quad u_i|_{\partial X} = g_i, \quad 1\le i \le m,
    \label{eq:elliptic}
\end{align}
where the vector fields $W_{ij}$ are known from the data and where  the illuminations $g_i$  were prescribed in the first place. 
Here and below, we use the Einstein convention of summation over repeated indices. 
The vector fields $W_{ij}$ satisfy stability conditions of the form
\begin{align}
    \|W\|_\infty &:= \max_{1\le i,j\le m} \|W_{ij}\|_{L^\infty(X)} \le C_W \|H\|_{W^{1,\infty}(X)}, \label{eq:stabW} \\
    \|W-W'\|_\infty &:= \max_{1\le i,j\le m} \|W_{ij}-W'_{ij}\|_{L^\infty}(X) \le C'_W \|H-H'\|_{W^{1,\infty}(X)}, \label{eq:stabWW}
\end{align}
whenever two data sets $H$ and $H'$ jointly satisfy condition \eqref{eq:positivity} with the same triple $(\O, \tau, c_0)$. After proving solvability of this system, one is able to reconstruct the functions $u_i$ and then to reconstruct $\sigma$ as described below. Uniqueness and stability of the solution to \eqref{eq:elliptic} with respect to the drift fields $W_{ij}$ relies on the fact that $-c_F^{-1} = -( (n-2)\alpha+1)$ is not an eigenvalue of the operator $\bfP_W:\H\to\H$ defined by
\begin{align}
    \bfP_W:\bfv\mapsto \bfP_W\bfv = [\bfP_W\bfv]_i \bfe_i = \Delta_D^{-1} (W_{ij}\cdot\nabla v_j)\ \bfe_i,
    \label{eq:PW}
\end{align}
where $\Delta_D^{-1}$ denotes the inverse of the Dirichlet Laplacian on $X$, and where we have defined the space $\H:= [H_0^1(X)]^m$, which makes $(\H, \|\cdot\|_\H)$ Hilbert once equipped with the norm 
\begin{align}
    \|\bfv\|_\H^2 = \sum_{i=1}^m \|v_i\|_{H_0^1}^2 = \sum_{i=1}^m \int_X |\nabla v_i|^2\ dx, \quad \bfv = (v_1,\dots,v_m).
    \label{eq:H}
\end{align}
When the coefficients $W_{ij}$ are bounded, we show that the operator $\bfP_W$ is compact and its operator norm satisfies the estimate $\|\bfP_W\|\le \sqrt{m} \|\Delta_D^{-1}\| \|W\|_\infty$ (see lemma \ref{lem:PW}), where $\|\Delta_D^{-1}\|$ denotes the operator norm of $\Delta_D^{-1}:L^2(X)\to H_0^1(X)$. As a consequence, the system \eqref{eq:elliptic} satisfies a Fredholm alternative which will provide uniqueness and stability as stated in the next proposition, for all $\alpha\in\Rm$ when $n=2$, and for all $\alpha$ but possibly a discrete set (possibly converging to $-(n-2)^{-1}$) in the interval $\left[ \frac{-\|\bfP_W\|-1}{n-2}, \frac{\|\bfP_W\|-1}{n-2} \right]$ whenever $n\ge 3$. 

\begin{proposition}[Stability of the strongly coupled elliptic system] \label{prop:elliptic}
    Let  vector fields $\{W_{ij}, W_{ij}'\}_{1\le i,j\le m}$ belong to $L^\infty(X)$ and such that $-c_F^{-1}$ is an eigenvalue of neither $\bfP_W$ nor $\bfP_{W'}$. Let $\bfu, \bfu'$ be the unique solutions to \eqref{eq:elliptic} with same illumination $\bfg$ and respective drift terms $W$, $W'$. Then we have that $\bfu-\bfu'\in\H$ and satisfies the stability estimate
    \begin{align}
	\|\bfu-\bfu'\|_\H \le C \|W-W'\|_\infty.
	\label{eq:stabelliptic}
    \end{align}
\end{proposition}
 
\begin{remark}
    In the case $n=2$ or ($\alpha=0$ with $m=n$), we can recast \eqref{eq:elliptic} as a coercive system in divergence form, the injectivity of which follows immediately. These cases correspond to $c_F = 1$.
\end{remark}

Once the solutions $u_i$ are reconstructed, one may reconstruct $\sigma$ using a formula of the form $\sigma = H_{11}/|\nabla u_1|^2$. However, such a formula may not offer the best stability estimates. Another reconstruction strategy is deduced from \eqref{eq:datatoF}, which can be recast locally as
\begin{align}
    \begin{split}
	\nabla (\sigma^{-2\alpha}) &= - \frac{2\alpha c_F}{D} \sum_{i,j=1}^n (\nabla (DH^{ij})\cdot\nabla u_i) \nabla u_j, \quad \alpha\ne 0,\quad (n-2)\alpha\ne -1, \\
	\nabla \log \sigma &= \frac{1}{D} \sum_{i,j=1}^n (\nabla (DH^{ij})\cdot\nabla u_i) \nabla u_j, \quad \alpha =0, \\
    \end{split}
    \label{eq:datatosigma}
\end{align}
where $\nabla u_1,\dots,\nabla u_n$ denote the $n$ linearly independent gradients. As in the ODE-based reconstruction procedure, we can devise an ODE-based algorithm to reconstruct $\sigma$ locally from formula \eqref{eq:datatosigma}.
We then arrive at the following stability result.
\begin{theorem}\label{thm:elliptic}
    Let the conditions of Proposition \ref{prop:elliptic} be satisfied. Then the corresponding $\sigma,\sigma'$ satisfy the estimate
    \begin{align}
	\begin{split}
	    \| \sigma^{-2\alpha} - \sigma'^{-2\alpha} \|_{H^1(X)} &\le C \|H-H'\|_{W^{1,\infty}(X)}, \quad \alpha\ne 0,\quad (n-2)\alpha\ne -1, \\
	    \| \log\sigma - \log\sigma' \|_{H^1(X)} &\le C \|H-H'\|_{W^{1,\infty}(X)}, \quad \alpha =0. \\	    	    
	\end{split}
	\label{eq:stabsigmaelliptic}
    \end{align}
\end{theorem}
Note that a necessary and sufficient condition for the unique solvability of \eqref{eq:datatosigma} on a simply connected domain is that the exterior derivative of the right-hand side (seen as a one-form) vanish by an application of the Poincar\'e lemma. The compatibility condition that arises here takes the form of a quadratic equation in the components of $\nabla u_i$ that is difficult to ensure as it depends on the unknowns. 

\section{Geometric setting and proofs of lemmas \ref{lem:datatoF} and \ref{lem:CGO}} \label{sec:datatoF}

Defining geometric notation for now, let us first denote the Euclidean orthonormal frame $\bfe_i = \partial_{x^i}$ and $\bfe^i = dx^i$. For $0\le k\le n$, $\Lambda^k$ denotes the space of $k-$ forms. We recall the definition of the \emph{Hodge star operator} $\star:\Lambda^k\to \Lambda^{n-k}$ for $0\le k\le n$, such that for any elementary $k$-form $dx^I = dx^{i_1}\wedge\dots\wedge dx^{i_k}$ with $I = (i_1,\dots,i_k)$, we have
\begin{align}
    \star dx^{I} = \sigma dx^{J}, \where\quad \sigma = \sgn{(1\dots n) \mapsto (I,J) }. 
    \label{eq:Hodge}
\end{align}
We recall the following useful identities, see e.g., \cite{T}:
\begin{align*}
    \star \star = (-1)^{k(n-k)} \quad\text{on } \Lambda^k, \quad \star (u^\flat\wedge\star v^\flat) = u\cdot v, \quad \star d \star u^\flat = \div u, \quad u,v\in\Lambda^1. 
\end{align*}
We now prove Lemma \ref{lem:datatoF} which is the cornerstone of our explicit reconstructions.

\begin{proof}[Proof of Lemma \ref{lem:datatoF}] Because $S_1(x),\dots,S_n(x)$ is a basis of $\Rm^n$ at any point $x\in X$, a vector $V$ can be represented in this basis by the following representation ($x$ is implicit here)
\begin{align}
    V = H^{ij} (V\cdot S_i) S_j.
    \label{eq:reprnd}
\end{align}
For $j=1,\dots,n$, let us introduce the following $1$-forms:
\begin{align}
     X^\flat_j := (-1)^{n-1} \sigma_j * ( S^\flat_{i_1}\wedge\dots\wedge S^\flat_{i_{n-1}}), \quad (i_1,\dots, i_{n-1}) = (1,\dots,\hat j,\dots,n),
    \label{eq:crossprodnd}
\end{align}
where the hat indicates an omission and $\sigma_j=(-1)^{j-1}$ is the signature of the permutation $(1,2\dots,n)\mapsto (j,1,\dots,j-1,j+1,\dots,n)$. At each $x\in\Omega$, the vector $X_j(x)$ obtained from $ X^\flat_j(x)$ by ``raising an index'' can also be seen as the unique vector obtained by the Riesz representation lemma that corresponds to the linear form $D_j:\Rm^n\to\Rm$ such that for any $V\in \Rm^n$,
\begin{align*}
    D_j (V) = \det(S_1(x), \dots, S_{j-1}(x), V, S_{j+1}(x), \dots S_n(x)) = X_j(x)\cdot V. 
\end{align*}
We now show that the vector fields $X_j$ satisfy a simple divergence equation. We compute
\begin{align*}
    \div X_j = \star d\star  X^\flat_j &= \sigma_j \star d ( S^\flat_{i_1}\wedge\dots\wedge S^\flat_{i_{n-1}}) \\
    &= \sigma_j \star \sum_{k=1}^{n-1} (-1)^k  S^\flat_{i_1}\wedge\dots\wedge d S^\flat_{i_k}\wedge\dots\wedge S^\flat_{i_{n-1}} \\
    &= \sigma_j \star \sum_{k=1}^{n-1} (-1)^k  S^\flat_{i_1}\wedge\dots\wedge \alpha( F^\flat\wedge S^\flat_{i_k})\wedge\dots\wedge S^\flat_{i_{n-1}} \\
    &= (n-1)\alpha \star ( F^\flat \wedge \star X^\flat_j),
\end{align*}
and using the identity $\star (u^\flat\wedge \star v^\flat) = u\cdot v$, we deduce
\begin{align}
    \div X_j = (n-1)\alpha F\cdot X_j, \quad j=1\dots n.
    \label{eq:divXj}
\end{align}
The decomposition of $X_j$ in the basis $S_1, \dots, S_n$ may be obtained by computing its dotproducts with $S_1,\dots,S_n$. Indeed, for $k\ne j$, there is an $l$ such that $i_l = k$ and we have
\begin{align*}
    X_j\cdot S_k = \det (S_1,\dots,S_{j-1}, S_k, S_{j+1}, \dots,S_n) = 0,
\end{align*}
by repetition of the term $S_k$ in the determinant. Now if $k=j$, we have
\begin{align*}
    X_j\cdot S_j = \det (S_1,\dots,S_n) = \det S = D.
\end{align*} 
Using formula \eqref{eq:reprnd}, we deduce that $X_j$ admits the expression
\begin{align*}
    X_j = D H^{ij} S_i.
\end{align*}
Plugging this expression into equation \eqref{eq:divXj}, and using  $\div (\varphi V) = \nabla\varphi\cdot V + \varphi\div V$, we obtain
\begin{align*}
    \nabla (D H^{ij})\cdot S_i + D H^{ij} \div S_i &= (n-1)\alpha F\cdot (D H^{ij} S_i) \\
    \Leftrightarrow \nabla (D H^{ij})\cdot S_i - DH^{ij} (1-\alpha) F\cdot S_i &= (n-1)\alpha D H^{ij} F\cdot S_i \\
    \Leftrightarrow \nabla (DH^{ij})\cdot S_i &= c_F^{-1} D H^{ij} F\cdot S_i.
\end{align*}
Finally using the representation \eqref{eq:reprnd} for $F$ itself yields
\begin{align}
    F = (H^{ij} F\cdot S_i) S_j = \frac{c_F}{D} (\nabla (DH^{ij})\cdot S_i) S_j.
    \label{eq:Fnd1}
\end{align}
We can also recast the previous expression as follows
\begin{align}
    F = c_F \left[ H^{ij} (\nabla\log D \cdot S_i) S_j + ((\nabla H^{ij})\cdot S_i) S_j \right] = c_F \left[ \nabla\log D +  ((\nabla H^{ij})\cdot S_i) S_j  \right], \label{eq:Fnd2}
\end{align}
and the proof is complete. 
\end{proof}

We now give a proof of Lemma \ref{lem:CGO}, which guarantees the existence of illuminations that ensure condition \eqref{eq:positivity} and thus justifies the two global reconstruction approaches. The CGO constructions, introduced in \cite{BU-IP-10} in this context, generalize those defined in \cite{BBMT}.

\begin{proof}[Proof of lemma \ref{lem:CGO}]
    Since $\sigma$ is bounded from above and below by positive constants, it suffices to study the case $\alpha = \frac{1}{2}$ since we have for any $\alpha_1, \alpha_2\in\Rm$, 
    \begin{align*}
	\det (\sigma^{\alpha_1}\nabla u_1,\dots, \sigma^{\alpha_1} \nabla u_n) = \sigma^{n (\alpha_1-\alpha_2)} \det (\sigma^{\alpha_2}\nabla u_1,\dots, \sigma^{\alpha_2} \nabla u_n).
    \end{align*}
    Consider the problem $\nabla\cdot\sigma(x)\nabla u=0$ on $\Rm^n$ with $\sigma(x)$ extended in a continuous manner outside of $X$ and such that $\sigma$ equals $1$ outside of a large ball. The construction requires sufficient smoothness of $\sigma$ in order to be valid. Let $q(x)=-\frac{\Delta\sqrt\sigma}{\sigma}$ on $\Rm^n$. We assume that $q\in H^{\frac n2+1+\ep}(\Rm^n)$, which holds if $\sigma-1\in H^{\frac n2+3+\ep}(\Rm^n)$ for some $\ep>0$, i.e., the original $\sigma_{|X}\in H^{\frac n2+3+\ep}(X)$. Note that by Sobolev imbedding, $\sigma$ is of class $\C^3(\overline X)$ while $q$ is of class $\C^1(\overline X)$. With the above hypotheses, we can apply \cite[Corollary 3.2]{BU-IP-10} which states the following. 
 
   Let $v=\sqrt\sigma u$ so that $\Delta v+qv=0$ on $\Rm^n$. Let $\bmrho\in\Cm^n$ be of the form $\bmrho = \rho(\bfk+i\bfk^\perp)$ with $\bfk, \bfk^\perp\in \Sm^{n-1},\ \bfk\cdot\bfk^\perp=0$, and $\rho = |\bmrho|/\sqrt{2}>0$. Thus, $\bmrho$ satisfies $\bmrho\cdot\bmrho=0$ and $e^{\bmrho\cdot x}$ is a harmonic complex plane wave (hence the name of complex geometrical optics solutions). Now, it is shown in \cite{BU-IP-10}, following works in \cite{calderon80,Syl-Uhl-87}, that 
   
   \begin{displaymath} 
       v_{\bmrho} = \sqrt\sigma u_{\bmrho} = e^{\bmrho\cdot x}(1+\psi_{\bmrho}),\qquad  \rho{\psi_{\bmrho}}_{|X} = O(1) \mbox{ in } \C^1(\overline X),
   \end{displaymath}
   with $(\Delta + q)v_{\bmrho}=0$ and hence $\nabla\cdot \sigma\nabla u_{\bmrho}=0$ in $\Rm^n$. We have used again the Sobolev imbedding stating that functions in $H^{\frac n2+k+\ep}(Y)$ are of class $\C^k(Y)$ for a bounded domain $Y$. Taking gradients of the previous equation and rearranging terms, we obtain that 
   \begin{align*}
       \sqrt{\sigma}\nabla u_{\bmrho} = e^{\bmrho\cdot x} (\bmrho + \bmphi_{\bmrho}), \quad\text{with}\quad \bmphi_{\bmrho} := \nabla\psi_{\bmrho} + \psi_{\bmrho} \bmrho - (1+\psi_{\bmrho}) \nabla\sqrt{\sigma}. 
   \end{align*}
   Because $\nabla \sqrt \sigma$ is bounded and $\rho{\psi_{\bmrho}}_{|X} = O(1)$ in $\C^1(\overline X)$, the $\Cm^n$-valued function $\bmphi_{\bmrho}$ satisfies $\sup_{\overline X} | \bmphi_{\bmrho}| \leq C$ independent of $\bmrho$. Moreover, the constant $C$ is in fact independent of $\sigma$ provided that the norm of the latter is bounded by a uniform constant in  $H^{\frac n2+3+\ep}(X)$. 
   
   Both the real and imaginary parts of $u_{\bmrho}$, denoted $u_{\bmrho}^\Re$ and $u_{\bmrho}^\Im$, count as solutions of the free-space conductivity equation, thus $\sqrt\sigma \nabla u_{\bmrho}^\Re$ and $\sqrt\sigma \nabla u_{\bmrho}^\Im$ can serve as vectors $S_i$. More precisely, we have 
   \begin{align}
       \begin{split}
	   \sqrt\sigma\nabla u_{\bmrho}^\Re  &= \rho e^{\rho\bfk\cdot x} \left( (\bfk+ \rho^{-1}\bmphi_{\bmrho}^\Re )\cos(\rho\bfk^\perp\cdot x) - (\bfk^\perp+ \rho^{-1}\bmphi_{\bmrho}^\Im ) \sin (\rho\bfk^\perp\cdot x) \right),  \\
	   \sqrt\sigma\nabla u_{\bmrho}^\Im  &= \rho e^{\rho\bfk\cdot x} \left( (\bfk^\perp + \rho^{-1}\bmphi_{\bmrho}^\Im) \cos(\rho\bfk^\perp\cdot x) + (\bfk + \rho^{-1} \bmphi_{\bmrho}^\Re) \sin (\rho\bfk^\perp\cdot x) \right). 
       \end{split}
       \label{eq:urho}    
   \end{align}

   \paragraph{\underline{Case $n$ even}} Set $n=2p$, define $\bmrho_l = \rho(\bfe_{2l}+ i\bfe_{2l-1})$ for $1\le l\le p$, and construct
   \begin{align*}
       S_{2l-1} = \sqrt{\sigma} \nabla u_{\bmrho_l}^\Re \qandq S_{2l} = \sqrt{\sigma} \nabla u_{\bmrho_l}^\Im, \quad 1\le l\le p.
   \end{align*}
   Using \eqref{eq:urho}, we obtain that 
   \begin{align*}
       \det(S_1,\dots,S_n) = \rho^n e^{2\rho\sum_{l=1}^p x_{2l}} (1 + f(x)),
   \end{align*}
   where $\lim_{\rho\to\infty} \sup_{\overline X} |f| = 0$. Letting $\rho$ so large that $\sup_{\overline X} |f|\le \frac{1}{2}$ and denoting \\ $\gamma_0:= \min_{x\in\overline X} (\rho^n e^{2\rho\sum_{l=1}^p x_{2l}})>0$, we have $\inf_{x\in\overline X} \det(S_1,\dots,S_n)\ge \frac{\gamma_0}{2}>0$. We conclude after the next paragraph. 

   \paragraph{\underline{Case $n$ odd}} Set $n=2p-1$, define $\bmrho_l = \rho(\bfe_{2l}+ i\bfe_{2l-1})$  for $1\le l\le p-1$, and $\bmrho_p = \rho(\bfe_n + i\bfe_{1})$ and construct
   \begin{align*}
       S_{2l-1} = \sqrt{\sigma} \nabla u_{\bmrho_l}^\Re \qandq S_{2l} = \sqrt{\sigma} \nabla u_{\bmrho_l}^\Im, \quad 1\le l\le p.
   \end{align*}   
   Using \eqref{eq:urho}, we obtain that 
   \begin{align*}
       \det(S_1,\dots, S_{n-1}, S_n) &= \rho^n e^{\rho(x_n + 2\sum_{l=1}^{p-1}x_{2l})} \left( - \cos(\rho x_1) + f_1(x) \right), \\
       \det(S_1,\dots, S_{n-1}, S_{n+1}) &= \rho^n e^{\rho(x_n + 2\sum_{l=1}^{p-1}x_{2l})} \left( -\sin(\rho x_1) + f_2(x) \right),
   \end{align*}   
   where $\lim_{\rho\to\infty} \sup_{\overline X} |f_i| = 0$ for $i=1,2$. Letting $\rho$ so large that $\sup_{\overline X} (|f_1|,|f_2|)\le \frac{1}{4}$ and denoting $\gamma_1 := \min_{x\in\overline X} (\rho^n e^{\rho(x_n + 2\sum_{l=1}^{p-1}x_{2l})}) >0$, we have that 
   \begin{align*}
       |\det(S_1,\dots,S_{n-1},S_n)|&\ge \frac{\gamma_1}{4}, \quad x\in X\cap \left\{ \rho x_1\in \left( \frac{-\pi}{3}, \frac{\pi}{3} \right) + m\pi\right\}, \\
       |\det(S_1,\dots,S_{n-1}, S_{n+1})|&\ge \frac{\gamma_1}{4}, \quad x\in X\cap \left\{ \rho x_1\in \left( \frac{\pi}{6}, \frac{5\pi}{6}  \right) + m\pi \right\},
   \end{align*}
   where $m$ is a signed integer. Since the previous sets are open and a finite number of them covers $X$ (because $X$ is bounded and $\rho$ is finite), we therefore have fulfilled the desired requirements of the construction. Upon changing the sign of $S_n$ or $S_{n+1}$ on each of these sets if necessary, we can assume that the determinants are all positive.  
	
   \paragraph{\underline{Conclusion}} In each of the previous cases, let $\{g_l\}_{1\le l\le m}$ be the traces of the solutions defined above with $m=2 \lfloor \frac{n+1}{2} \rfloor$. These illuminations generate solutions that satisfy the desired properties of maximal rank and positive determinants. By continuity arguments, any boundary conditions $\tilde g_l$ in an open set sufficiently close to $g_l$ will ensure that the maximum of the determinants stay bounded from below by $c_0>0$. This concludes the proof of the lemma.
\end{proof}

\section{The ODE-based method} \label{sec:ODE}
In this section, we extend the results presented in \cite{BBMT} to general dimension and for a more general class of measurements (described by the coefficient $\alpha$). We first need to introduce standard geometric notation, without which the derivations become quickly intractable. 

\subsection{Definitions, notation and identities}

We work on a convex set $\Omega\subset\Rm^n$ with the Euclidean metric $g(X,Y) \equiv X\cdot Y = \delta_{ij} X^i Y^j$ on $\Rm^n$. Following \cite{L}, we denote by $\del$ the Euclidean connection, i.e. the unique connection that is torsion-free, and compatible with the Euclidean metric in the sense that
\begin{align*}
    \del_X (Y\cdot Z) = (\del_X Y)\cdot Z + Y\cdot (\del_X Z), 
\end{align*}
for smooth vector fields $X,Y,Z$. On zero- and one-forms, this connection takes the expression:
\begin{align*}
    \del_X f = X\cdot\nabla f = X^i \partial_i f, \qandq \del_X Y = (X\cdot \nabla Y^j) \bfe_j = X^i(\partial_i Y^j) \bfe_j, 
\end{align*}
for given vector fields $X = X^i \bfe_i$ and $Y = Y^i\bfe_i$. An important identity for the sequel is the following characterization of the exterior derivative of a one-form $\omega$
\begin{align}
    d\omega(X,Y) = \del_X (\omega(Y)) - \del_Y (\omega(X)) - \omega([X,Y]),
    \label{eq:dcarac}
\end{align}
or equivalently in the Euclidean metric, writing $\omega = Z^\flat$ for some vector field $Z$, 
\begin{align}
    Z\cdot[X,Y] = \del_X (Z\cdot Y) - \del_Y (Z\cdot X) - dZ^\flat(X,Y),
    \label{eq:dcarac2}    
\end{align}
where the Lie bracket (commutator) of $X$ and $Y$ coincides with (and thus may be ``defined'' here as) $[X,Y]=\del_X Y-\del_Y X$ by virtue of the torsion-free property.

A \emph{frame} refers to an oriented family $E = (E_1,\dots,E_n)$ of $n$ vector fields over $\Omega$ such that for every $x\in\Omega$, $(E_1(x),\dots E_n(x))$ is a basis of $T_x \Omega \equiv \Rm^n$. For a given frame $E$, we define the \emph{Christoffel symbols} (of the second kind) with respect to this frame, by the relations 
\begin{align}
    \begin{split}
	\del_{E_i} E_j &= \Gamma_{ij}^k E_k, \quad\text{i.e.}\quad \Gamma_{ij}^q = g^{pq}\del_{E_i} E_j \cdot E_p, \where \\
	g_{ij}&=E_i\cdot E_j \qandq g^{pq}=(g^{-1})_{pq}.	
    \end{split}    
    \label{eq:chrisE}
\end{align}
The following very useful identity allows us to compute the Christoffel symbols from inner products and Lie brackets of a given frame (see e.g. \cite[Eq. 5.1 p. 69]{L}):
\begin{align}
    \begin{split}
	2 (\del_X Y)\cdot Z &=  \del_X (Y\cdot Z) + \del_Y (Z\cdot X) - \del_Z (X\cdot Y)\\
	&\quad  - Y\cdot [X,Z] - Z\cdot[Y,X] + X\cdot[Z,Y],	
    \end{split}    
    \label{eq:chriscarac}
\end{align}
where $X,Y,Z$ are smooth vector fields. 

For a vector $X=X^j\bfe_j$, we want to form the matrix of partial derivatives $(\partial_j X^i)_{i,j}$. Geometrically, gradients generalize to tensors via the total covariant derivative, which maps a vector field $X$ to a tensor of type $(1,1)$ defined by
\begin{align}
    \del X (\omega, Y) = \omega (\del_Y X). 
    \label{eq:totcovder}
\end{align}
In a given frame $E$, we may express $\del E_i$ in the basis $\{E_j \otimes E_k^\flat\}_{j,k=1}^n$ of such tensors by writing $\del E_i = a_{ijk} E_j \otimes E_k^\flat$ and identifying the coefficients $a_{ijk}$ by writing 
\begin{align*}
    \del E_i (E_p^\flat ,E_q) = E_p^\flat (\del_{E_q} E_i) = \del_{E_q} E_i \cdot E_p = g_{pr} \Gamma_{qi}^{r},
\end{align*}
and also 
\begin{align*}
    \del E_i (E_p^\flat, E_q) = a_{ijk} E_j \otimes E_k^\flat (E_p^\flat, E_q) = a_{ijk} g_{jp} g_{kq}. 
\end{align*}
Equating the two, we obtain the representation 
\begin{align}
    \del E_i = g^{qk} \Gamma_{qi}^j E_j \otimes E_k^\flat = g^{qk} g^{jp} (\del_{E_q} E_i \cdot E_p) E_j \otimes E_k^\flat.  
    \label{eq:totcovderE}
\end{align}

The theory of the following sections proves that all partial derivatives of a frame (given in \eqref{eq:totcovderE}) are uniquely determined by inner products $g_{ij}$ and by Lie brackets, as \eqref{eq:chriscarac} indicates, or equivalently by exterior derivatives, as \eqref{eq:dcarac2} expresses. These derivations will be carried out first for the $S$ frame and second for the $R$ frame with values in the space of rotations $SO(n,\Rm)$.

\subsection{The $S$ frame}

We now study the properties of the $S$ frame. $S$ is a frame provided that the determinant condition $\inf_{x\in\Omega}\det S\ge c_0>0$ holds. Our objective in this section is to find an expression for $\del S_i$ that allows us to solve for $S_i$ by the method of characteristics. We have seen in the preceding section that this involved calculating the Lie brackets (commutators) of the vectors composing the frame. For $1\le i<j\le n$, we have
\begin{align}
    [S_i,S_j] = H^{kl} ([S_i,S_j]\cdot S_k) S_l.
    \label{eq:Lie1}
\end{align}
Now using \eqref{eq:dcarac2} we write
\begin{align*}
    S_k\cdot [S_i,S_j] &= \del_{S_i} ( S_k\cdot S_j ) - \del_{S_j} (S_k\cdot S_i) - d S^\flat_k (S_i,S_j) \\
    &= S_i\cdot\nabla H_{kj} - S_j\cdot\nabla H_{ki} - \alpha F^\flat\wedge S^\flat_k (S_i,S_j) \\
    &= S_i\cdot\nabla H_{kj} - S_j\cdot\nabla H_{ki} + \alpha ( - H_{kj} F\cdot S_i + H_{ki} F\cdot S_j). 
\end{align*}
Plugging this into \eqref{eq:Lie1} and using that $H^{kl}H_{kj} = \delta_{lj}$, we obtain the Lie brackets $[S_i,S_j]$ for $1\le i< j\le n$ :
\begin{align}
    [S_i,S_j] = H^{kl}[\nabla H_{jk} \cdot S_i - \nabla H_{ik}\cdot S_j] S_l + \alpha( (F\cdot S_j) S_i - (F\cdot S_i) S_j).
    \label{eq:LieSi}
\end{align}
Returning to the computation of $\del S_i$ using \eqref{eq:totcovderE}, we combine \eqref{eq:LieSi} with \eqref{eq:chriscarac} to arrive at
\begin{align*}
    2 (\del_{S_q} S_i)\cdot S_p &= \del_{S_q} H_{ip} + \del_{S_i} H_{pq} - \del_{S_p} H_{qi} \\
    &\quad - S_i\cdot [S_q,S_p] - S_p\cdot[S_i,S_q] + S_q\cdot[S_p,S_i] \\
    &= \nabla H_{iq}\cdot S_p + \nabla H_{ip}\cdot S_q - \nabla H_{pq}\cdot S_i + 2 \alpha ( H_{pq} (F\cdot S_i) - H_{qi}(F\cdot S_p)).
\end{align*}
Plugging this expression into \eqref{eq:totcovderE} (expressed in the $S$ frame), and using $H_{ij}H^{jk} = \delta_{ik}$, we obtain
\begin{align*}
    2\del S_i &= 2 H^{qk} H^{jp} (\del_{S_q} S_i \cdot S_p) S_j \otimes S_k^\flat \\
    &= H^{qk} H^{jp}\big( \nabla H_{iq}\cdot S_p + \nabla H_{ip}\cdot S_q - \nabla H_{pq}\cdot S_i \\
    &\qquad + 2\alpha( H_{pq} (F\cdot S_i) - H_{qi}(F\cdot S_p) ) \big) S_j \otimes S_k^\flat \\
    &= \big( H^{jp} U_{ik}\cdot S_p + H^{qk} U_{ij}\cdot S_q + \nabla H^{jk}\cdot S_i \\
    &\qquad + 2 \alpha ( H^{jk} (F\cdot S_i) - H^{jp} \delta_{ik} (F\cdot S_p)) ) S_j \otimes S_k^\flat,
\end{align*}
where we have used $\nabla H^{jk} = - H^{jp} (\nabla H_{pq}) H^{qk}$ and  have defined 
\begin{align}
    U_{jk} := (\nabla H_{jp})H^{pk} = -H_{jp} \nabla H^{pk}, \quad 1\le j,k\le n.
    \label{eq:Ujk}
\end{align}

Using formulas $H^{jk} S_j\otimes S^\flat_k = \Imm_n := \bfe_i \otimes \bfe^i$ and $H^{kl} (V\cdot S_k) S_l = V$ for any smooth vector field $V$, we obtain for $1\le i\le n$
\begin{align}
    \del S_i = \frac{1}{2} \left( U_{ik} \otimes S_k^\flat + S_k \otimes U_{ik}^\flat + (\nabla H^{jk} \cdot S_i) S_j\otimes  S^\flat_k \right) + \alpha(F\cdot S_i) \Imm_n - \alpha F\otimes  S^\flat_i. 
    \label{eq:gradSi}
\end{align}
Using \eqref{eq:Fnd1}, we observe that $\del S_i$ is equal to a polynomial of degree at most three in the frame $S$ with coefficients involving the known inner products $H_{ij}$.
For each $1\le i,k\le n$, $\partial_k S_i$ is nothing but $\del_{\bfe_k} S_i = \del S_i (\cdot, \bfe_k)$, which can be obtained from \eqref{eq:gradSi}. Denoting $\bfS:= [S_1|\dots|S_n]$, we are then able to construct the system of equations
\begin{align}
    \partial_k \bfS = \sum_{|\beta|\le3} Q_\beta^k \bfS^\beta, \qquad \bfS^\beta = \prod_{i=1}^{n^2} \bfS_i^{\beta_i} , \quad 1\le k\le n,
    \label{eq:ODES}
\end{align}
where $Q_\beta^k$ depends only on the data and $\beta$ is an $n^2$-index. This redundant system can then be integrated along any curve (where it becomes a system of ordinary differential equations with Lipschitz right-hand sides ensuring uniqueness of the solution) in order to solve for the matrix-valued function $\bfS$.  

\subsection{The orthonormal $R$ frame}
\label{sec:Rframe}

The above system \eqref{eq:ODES} involves a priori $n^2$ unknowns since the matrix $S$ does not necessarily have any useful symmetries. However, we know the inner products $H=S^TS$, i.e., a matrix of dimension $\frac12 n(n+1)$. We therefore hope to be able to find a closed-form system involving $\frac12 n(n-1)$ dimensions. This is the dimension of the orthonormal $R$ frame.

We now provide the details of remark \ref{rem:Rframe}. From the frame $S$, we build an oriented orthonormal frame $R = [R_1|\dots|R_n]$ (or equivalently, an $SO_n(\Rm)$-valued function) from a matrix-valued function $T(x) = \{t_{ij}(x)\}_{1\le i,j\le n}$ that satisfies the relations $T^T T = H^{-1}$ and $\det T>0$ at every $x\in\Omega$, as well as a stability property of the form 
\begin{align}
    \|T-T'\|_{W^{1,\infty}(X)} \le C_T \|H-H'\|_{W^{1,\infty}(X)},
    \label{eq:stabT}
\end{align}
where $C_T>0$ depends only on the way we construct $T$ from $H$. $T$ can either be constructed by the GS procedure or by setting $T = H^{-\frac{1}{2}}$, the positive square root of $H^{-1}$. The stability statement \eqref{eq:stabT}, first proved in the GS case for $n=2,3$ in \cite{BBMT}, can be obtained for both GS and $T= H^{-\frac{1}{2}}$, see \cite{M} for proofs of these statements. 

The function $R:= ST^T$ satisfies everywhere $R^T R = \Imm_n$ and $\det R = 1$, hence $R$ is an $SO_n(\Rm)$-valued function. The column vectors of $S$ and $R$ transform according to:
\begin{align}
    R_i = t_{ij}S_j, \quad S_i = t^{ij} R_j, \quad i=1\dots n.
    \label{eq:RtoS}
\end{align}
We also define for $1\le i,k\le n$
\begin{align}
    V_{ik} := (\nabla t_{ij}) t^{jk}, \quad V_{ik}^s := \frac{1}{2}(V_{ik}+V_{ki}) \qandq V_{ik}^a := \frac{1}{2}(V_{ik}-V_{ki}).
    \label{eq:Vik_nd}
\end{align}
We are brief on the derivation of the gradient system for $R$ as it is very similar to that of the $S$ frame. The system of equations \eqref{eq:divSj}-\eqref{eq:dSj} together with the transformation rules \eqref{eq:RtoS} allow us to derive the following system of equations for the $R$ frame:
\begin{align}
    \div R_i &= V_{ik}\cdot R_k - (1-\alpha) F\cdot R_i, \label{eq:divRi} \\
    d R^\flat_i &= V^\flat_{ik}\wedge  R^\flat_k + \alpha F^\flat\wedge  R^\flat_i, \quad 1\le i\le n. \label{eq:dRi}
\end{align}
From this system, we express $F$ in the $R$ frame as
\begin{align}
    F = c_F \left( \nabla\log D + ((V_{ij}+V_{ji})\cdot R_i)R_j \right).
    \label{eq:RtoF}
\end{align}
Equation \eqref{eq:RtoF} can also be derived directly from \eqref{eq:datatoF} and the transformation rules \eqref{eq:RtoS}. Then, using equation \eqref{eq:dRi} and formula \eqref{eq:dcarac2}, the Lie brackets $[R_i,,R_j]$ of the vectors take the form, for  $1\le i<j\le n$:
\begin{align}
    [R_i,R_j] = (-V_{pj}\cdot R_i + V_{pi}\cdot R_j) R_p + \alpha((F\cdot R_j)R_i - (F\cdot R_i)R_j).
    \label{eq:LieRi}
\end{align}
From \eqref{eq:LieRi} we deduce the Christoffel symbols relative to the $R$ frame:
\begin{align}
    \Gamma_{ij}^k = V_{jk}^a\cdot R_i + V_{ik}^s\cdot R_j - V_{ij}^s\cdot R_k + \alpha(F\cdot R_j) \delta_{ik} - \alpha(F\cdot R_k) \delta_{ij}.
    \label{eq:chrisR}
\end{align}
Finally, in the orthonormal case, the expression of the gradient reduces to $\del R_i = \Gamma_{ki}^j R_j\otimes R^\flat_k$, from which we deduce that
\begin{align}
    \del R_i = R_k\otimes V^{a\flat}_{ik} - V_{ik}^s\otimes R^\flat_k + (V_{jk}^s\cdot R_i) R_j\otimes R^\flat_k + \alpha(F\cdot R_i) \Imm_n - \alpha F\otimes R^\flat_i. 	
    \label{eq:gradRi}
\end{align}
As for the $S$ frame, the R.H.S. of \eqref{eq:gradRi} depends polynomially on $R$ and on the data. This system can thus be solved for the vectors $R_i$ via ODE integration along any curve in a  connected domain and provided that we know the $R$ frame at one point. In practice, this system is less expensive to integrate than \eqref{eq:ODES} since the $R$ frame can be locally parameterized with $n(n-1)/2$ scalar functions (such as the Euler angles) whereas the $S$ frame requires $n^2$ scalar functions. 

\subsection{Global reconstruction algorithm}

The proof of the stability theorem \ref{thm:ODEstab} can be found in \cite{BBMT} in dimension $n=3$ with $\alpha=\frac{1}{2}$ (although the proof would be identical in arbitrary dimension). In that paper, the theorem is proved using the system for the rotation matrix $R$ and thus requires the extra stability condition \eqref{eq:stabT}. This condition is necessary only if we reconstruct $\sigma$ via the $R$ frame. The same stability result can be obtained without this requirement if we reconstruct $\sigma$ via the $S$ frame directly. In the latter setting, the proof is quite similar to the one in \cite{BBMT} with the further simplification that we do not need to change bases when switching subdomain $\Omega_i$. The system of ODEs that one must solve based on the gradient system \eqref{eq:gradSi} is well-posed since the function $\bfS$ satisfies a priori the uniform bound
\begin{align*}
    |\bfS(x)|^2 = \sum_{i=1}^m H_{ii} \le m \|H\|_\infty, 
\end{align*}
and the right-hand side of \eqref{eq:ODES} is Lipschitz in $\bfS$ over the set $\{\bfS:\overline X\to \Rm^{nm}, \|\bfS\|_\infty\le \sqrt{m \|H\|_\infty} \}$ as a polynomial of the components of $\bfS$, and using the fact that the polynomial $Q_\beta^k$ are bounded; see \cite{BBMT} for additional details, which we do not reproduce here.  

\section{The elliptic method} \label{sec:elliptic}

\subsection{Derivation of system \eqref{eq:elliptic} and equivalent formulations}

\subsubsection{The case $m=n$}
In this case, condition \eqref{eq:positivity} is satisfied with the partition $\O = \{X\}, N=1$. Equation \eqref{eq:Fnd1} can be rewritten as 
\begin{align}
    \nabla\log\sigma = \frac{c_F}{D} \sigma^{2\alpha} (\nabla (DH^{kl})\cdot \nabla u_k ) \nabla u_l.
    \label{eq:Fnd3}
\end{align}
Rewriting the conductivity equation \eqref{eq:conductivity} as
\begin{align*}
    \Delta u_i + \nabla\log\sigma\cdot \nabla u_i = 0, 
\end{align*} 
and plugging \eqref{eq:Fnd3} into it yields the coupled elliptic system of equations
\begin{align}
    0 = \Delta u_i + \frac{c_F}{D} (\nabla (DH^{kl})\cdot \nabla u_k ) \sigma^{2\alpha} \nabla u_l\cdot \nabla u_i = \Delta u_i + c_F W_{ik}\cdot\nabla u_k,
    \label{eq:ellipticsys}
\end{align}
where we have have defined
\begin{align}
    W_{ik} := \frac{H_{il}}{D}\nabla (DH^{lk}) = \nabla\log D \delta_{ik} + H_{il}\nabla H^{kl}, \quad 1\le i,k\le n.
    \label{eq:Wik}
\end{align}
From the last form of $W_{ik}$, we derive \eqref{eq:stabW} and \eqref{eq:stabWW} since the denominators only involve $D$ which is bounded away from zero and the rest is polynomial in the $H_{ij}$'s and their derivatives. 

Multiplying \eqref{eq:ellipticsys} by $DH^{pi}$ and writing it in divergence form, one obtain the following equivalent formulation to \eqref{eq:ellipticsys} in variational form:
\begin{align}
    -\nabla\cdot (DH^{pi} \nabla u_i) + \left( 1- c_F \right) \nabla (D H^{pi})\cdot \nabla u_i = 0, \quad 1\le p\le n.
    \label{eq:ellipticsys3}
\end{align}

\subsubsection{The case $m>n$}
In the case where we have $m>n$ solutions, we can still define $\nabla\log\sigma$ over the entire domain $X$ using a partition of unity that is subordinate to the open cover $\O$, call it $\{\varphi_i\}_{i=1}^N$. Then we can define $\nabla\log\sigma$ globally over $X$ by writing 
\begin{align*}
    \nabla\log\sigma = \sum_{i=1}^N \nabla\log\sigma |_{\Omega_i} \varphi_i , 
\end{align*}
where the restrictions are constructed from the $n$ solutions of positive determinant on each $\Omega_i$. Each of these restrictions can still be written in the form 
\begin{align*}
    \nabla\log\sigma|_{\Omega_i} &= c_F \sigma^{2\alpha} \sum_{j,k=1}^m (F_{jk}|_{\Omega_i}\cdot \nabla u_j) \nabla u_k, \where \\
    F_{jk}|_{\Omega_i} &= \left\{
    \begin{array}{ll}
	0  & \text{if } j\notin \tau(i) \text{ or } k\notin\tau(i) \\
	\frac{1}{D^{\tau(i)}} \nabla \left( D^{\tau(i)} H^{\tau(i),-1}_{ab} \right) & \text{if } (j,k) = (\tau(i)_a,\tau(i)_b),
    \end{array}
    \right.
\end{align*}
with $D^{\tau(i)} = \sqrt{\det H^{\tau(i)}}$. Thus we can patch these formulas together into a globally defined
\begin{align*}
    \nabla\log\sigma := c_F \sigma^{2\alpha} \sum_{j,k=1}^m (F_{jk}\cdot \nabla u_j) \nabla u_k, \where\quad  F_{jk} = \sum_{i=1}^N F_{jk}|_{\Omega_i}\varphi_i. 
\end{align*}
Plugging this expression into the conductivity equation yields the coupled elliptic system
\begin{align}
    \begin{split}
	0 &= \Delta u_i + \nabla\log\sigma\cdot\nabla u_i = \Delta u_i + c_F H_{ik} F_{jk}\cdot\nabla u_j. \\
	u_i|_{\partial X} &= g_i, \qquad 1\le i\le m,
    \end{split}
    \label{eq:ellipticsysm}    
\end{align}
and one arrives at a system of the form \eqref{eq:elliptic} by setting $W_{ij} := H_{ik} F_{jk}$ for every $1\le i,j\le m$. In this case, the stability inequalities \eqref{eq:stabW} and \eqref{eq:stabWW} can be derived using the fact that 
\begin{align*}
    \|W\|_{L^\infty(X)} \le \max_{1\le i\le N} \|W\|_{L^\infty(\Omega_i)},
\end{align*}
and noticing that on each $\Omega_i$, $W_{ij}$ is either zero or locally defined by \eqref{eq:Wik} (and weighed by $\varphi_i$) whose expression has been proved to be stable. A similar argument holds for proving \eqref{eq:stabWW} thanks to the fact that the partition of unity $\{\varphi_i\}$ is the same for two data sets $H, H'$ that jointly satisfy \eqref{eq:positivity} with the same triple $(\O,\tau,c_0)$.

\subsection{Uniqueness and stability results}

\subsubsection{Proofs of Proposition \ref{prop:elliptic} and Theorem \ref{thm:elliptic}}
Let us assume a system of the form \eqref{eq:ellipticsysm}, where the vector fields $W_{ij}$ belong to $L^\infty(X)$. Assuming the illumination $\bfg$ to be in $[H^{\frac{1}{2}}(\partial X)]^m$, we use a lifting operator to define functions $\{w_i\}_{i=1}^m \in [H^1(X)]^m$ of traces $\bfg$ at $\partial X$. Defining the unknown $v_i = u_i - w_i$, we are now left with analyzing the solvability of the system
\begin{align}
    \Delta v_i + c_F W_{ij}\cdot \nabla v_j &= h_i \quad (X), \quad v_i|_{\partial X} = 0, \quad 1\le i\le m, \label{eq:zerocond}\\
    \where\quad h_i &:= \Delta w_i + c_F W_{ij}\cdot\nabla w_j \in H^{-1}(X),  \label{eq:hi} 
\end{align}
as well as the stability of its solution with respect to the vector fields $W_{ij}$. $H^{-1}(X)$ denotes the dual space of $H^1_0(X)$. 

As described in section \ref{sec:results}, we apply the inverse of the Dirichlet Laplacian to \eqref{eq:zerocond} and obtain the system of integral equations
\begin{align*}
    v_i + c_F \Delta_D^{-1} (W_{ij}\cdot \nabla v_j) = \Delta_D^{-1} h_i, \quad 1\le i\le m,
\end{align*}
which can be recast in vector notation as 
\begin{align}
    \begin{split}
	(\bfI + c_F \bfP_W)\bfv &= \bff, \where \\
	\bfP_W\bfv &:= [\bfP_W\bfv]_i \bfe_i = \Delta_D^{-1} (W_{ij}\cdot \nabla v_j) \bfe_i, \qandq \bff := \Delta_D^{-1} h_i\ \bfe_i.
    \end{split}
    \label{eq:integralsys}
\end{align}
Because $\Delta_D^{-1}$ is continuous in the functional setting $H^{-1}(X)\to H^1_0(X)$ (see \cite{E}), it is clear that $\bff$ belongs to the space $\H$. 
We now have the following:
\begin{lemma} \label{lem:PW}
    Assuming that the vector fields $W_{ij}\in L^\infty(X)$, the operator $\bfP_W:\H\to\H$ defined in \eqref{eq:integralsys} is compact, and its norm satisfies
    \begin{align}
	\|\bfP_W\| \le \sqrt{m} \|\Delta_D^{-1}\| \|W\|_\infty , \quad \|W\|_\infty = \max_{1\le i,j\le m} \|W_{ij}\|_\infty.
	\label{eq:PWnorm}	
    \end{align}
\end{lemma}
\begin{proof}
    As can be seen in \cite{E} for instance, the operator $\Delta_D^{-1}: L^2(X)\to H^2(X)$ is bounded.  Therefore, by the Rellich compactness theorem, the operator $\Delta_D^{-1}: L^2(X)\to H_0^1(X)$ is compact and of norm denoted by $\|\Delta_D^{-1}\|$. Now $P$ is also compact since each of its components is the composition of the continuous operator $\H\ni\bfv\mapsto W_{ij}\cdot\nabla v_j \in L^2(X)$ with the compact operator $\Delta_D^{-1}: L^2(X)\to H_0^1(X)$. Moreover, for $\bfv\in\H$ and every $1\le i,j\le m$, we have the obvious bounds
    \begin{align*}
	\|\Delta_D^{-1}(W_{ij}\cdot\nabla v_j)\|_{H_0^1} \le \|\Delta_D^{-1}\| \|W_{ij}\cdot\nabla v_j\|_{L^2} \le \|\Delta_D^{-1}\| \|W\|_\infty \|v_j\|_{H_0^1},
    \end{align*}
    and thus
    \begin{align*}
	\|[\bfP_W\bfv]_i\|_{H_0^1}^2 \le \|\Delta_D^{-1}\|^2 \|W\|_\infty^2 \|\bfv\|_\H^2.
    \end{align*}
    Summing over $i$ proves \eqref{eq:PWnorm}. The proof is complete.
\end{proof}

As a consequence of lemma \ref{lem:PW} and by virtue of standard compact operator theory (e.g. \cite[Theorem 6 p 643]{E}), we have the following facts:
\begin{itemize}
    \item $0$ is eigenvalue of $\bfP_W$, which corresponds to the case $(n-2)\alpha=-1$, a value for $\alpha$ that we exclude from our analysis,
    \item the remaining spectrum of $P_W$ is point spectrum and consists of at most a discrete sequence of values that is either finite or converges to zero.
\end{itemize}
Finally, the operator $\bfI + c_F \bfP_W \in \L(\H)$ satisfies a Fredholm alternative. Therefore it suffices that $-c_F^{-1} \notin\sp(\bfP_W)$ in order to obtain uniqueness and stability of the solution of \eqref{eq:integralsys} and therefore of the solution of \eqref{eq:elliptic} as well. The proof of Proposition \ref{prop:elliptic} makes these statements more precise. 

\begin{proof}[Proof of Proposition \ref{prop:elliptic}]
    Let $W, W'$ have their coefficients in $L^\infty(X)$ and such that $-c_F^{-1} \notin \sp (\bfP_W)\cup\sp (\bfP_{W'})$, and let $\bfv, \bfv'\in\H$ solve the system \eqref{eq:zerocond} with respective drift terms $W$, $W'$ and same illumination $\bfg$. By virtue of the Fredholm alternative, the operators $\bfI+ c_F \bfP_W$ and $\bfI+c_F \bfP_{W'}$ are invertible with continuous inverses in $\L(\H)$.    
    
    Applying the inverse Dirichlet Laplacian to both systems, we obtain the systems
    \begin{align*}
	(\bfI+ c_F \bfP_W)\bfv = \bff, \qandq (\bfI+c_F \bfP_{W'})\bfv' = \bff'.
    \end{align*}
    Taking the difference of both systems, the resulting system reads
    \begin{align}
	(\bfI + c_F \bfP_W) (\bfv-\bfv') = \bff - \bff' - c_F \bfP_{W-W'} \bfv'.
	\label{eq:diffs}
    \end{align}
    The first difference in the right-hand side of \eqref{eq:diffs} may be bounded by 
    \begin{align*}
	\|\bff-\bff'\|_\H = \| c_F \Delta_D^{-1} \left( (W_{ij}-W_{ij}')\cdot\nabla w_j \right)\ \bfe_i \|_\H \le C c_F \|W-W'\|_{\infty},
    \end{align*}
    where the constant $C$ depends on $\|\Delta_D^{-1}\|_{\L(H^{-1}, H^1_0)}$ and $\max_{1\le i\le m} \|g_i\|_{H^{\frac{1}{2}}(X)}$. Applying lemma \ref{lem:PW} to the operator $\bfP_{W-W'}$, the second difference in the right-hand side of \eqref{eq:diffs} may be bounded by 
    \begin{align*}
	\|\bfP_{W-W'} \bfv'\|_\H &\le \sqrt{m} \|W-W'\|_\infty \|\Delta_D^{-1}\| \|\bfv'\|_\H \\
	& \le \sqrt{m} \|W-W'\|_\infty \|\Delta_D^{-1}\| \|(\bfI+c_F \bfP_{W'})^{-1}\| \|\bff\|_\H.
    \end{align*}
    Combining the last two estimates with \eqref{eq:diffs} and the fact that $\bfI+ c_F \bfP_W$ is invertible with continuous inverse in $\L(\H)$, we arrive at 
    \begin{align*}
	\|\bfv-\bfv'\|_\H \le C' \|(\bfI+c_F \bfP_W)^{-1}\|  \|W-W'\|_\infty,
    \end{align*}
    for some constant $C'>0$. Since $\bfu-\bfu' = \bfv-\bfv'$, this concludes the proof.
\end{proof}

We now conclude with the proof of theorem \ref{thm:elliptic}. 

\begin{proof}[Proof of theorem \ref{thm:elliptic}]
    We focus on the case $\alpha\ne 0$ and $(n-2)\alpha\ne -1$. The proof for $\alpha=0$ is identical up to small changes in notation. Let $H,H'$ have their components in $W^{1,\infty}(X)$ and jointly satisfy \eqref{eq:positivity} with the same triple $(\O,\tau,c_0)$. Then the families of vector fields $W$ and $W'$ have their coefficients in $L^\infty(X)$ and we further assume that $-c_F^{-1} \notin\sp(\bfP_W)\cup\sp(\bfP_{W'})$. Let $\bfv,\bfv'\in\H$ solve the system \eqref{eq:zerocond} with respective drift terms $W$ and $W'$ and same illumination $\bfg$, and let $\sigma,\sigma'$ be the corresponding conductivities.
    Without loss of generality, we work on one of the open sets $\Omega_i\in\O$ and renumber the $n$ solutions whose gradients are linearly independent from $1$ to $n$. The result will then hold provided that we have $X\subset\cup_{i=1}^N\Omega_i$ and thus
    \begin{align*}
	\|\sigma^{-2\alpha} - \sigma'^{-2\alpha}\|_{H^1(X)}^2 \le \sum_{i=1}^N \|\sigma^{-2\alpha} - \sigma'^{-2\alpha}\|_{H^1(\Omega_i)}^2.
    \end{align*}
    For a given $\Omega_i\in \O$, and defining $V_{ij}:= -2\alpha\frac{c_F}{D}\nabla(D H^{ij})$, we write, using equality \eqref{eq:datatosigma} 
    \begin{align}
	\begin{split}
	    \nabla (\sigma^{-2\alpha} - \sigma'^{-2\alpha}) &= ((V_{ij}-V'_{ij})\cdot\nabla u_i) \nabla u_j + (V'_{ij}\cdot\nabla(u_i-u'_i))\nabla u_j \\
	    &\qquad + (V'_{ij}\cdot\nabla u'_i) \nabla (u_j- u'_j).	    
	\end{split}	
	\label{eq:triangle}
    \end{align}
    Similarly to the vector fields $W_{ij}$ \eqref{eq:Wik}, the vector fields $V_{ij}$ satisfy estimates of the form 
    \begin{align}
	\|V_{ij}\|_\infty \le C_V \|H\|_{W^{1,\infty}} \qandq \|V_{ij}- V'_{ij}\|_\infty \le C'_V \|H-H'\|_{W^{1,\infty}}, \quad 1\le i,j\le n.
	\label{eq:stabV}
    \end{align}
    Since $H$ is bounded and $\sigma,\sigma'$ are assumed to be bounded from below by a constant $\sigma_0>0$, each of the $\nabla u_i$ is uniformly bounded by $\sqrt{H_{ii}/\sigma_0}\le \sqrt{\|H\|_\infty/\sigma_0}$. Taking $L^2$ norms over $\Omega_i$ and using the triangle inequality in \eqref{eq:triangle}, we obtain 
    \begin{align*}
	\|\nabla (\sigma^{-2\alpha} - \sigma'^{-2\alpha})\|_{L^2(\Omega_i)} &\le \|\nabla u_i\|_\infty \|\nabla u_j\|_\infty \|V_{ij}-V'_{ij}\|_{L^2} \\
	&\qquad + 2\|V'_{ij}\|_\infty \|\nabla (u_i-u'_i)\|_{L^2} \|\nabla u_j\|_\infty,
    \end{align*}
    which by virtue of proposition \ref{prop:elliptic} and estimates \eqref{eq:stabV} yields an estimate of the form
    \begin{align}
	\|\nabla (\sigma^{-2\alpha} - \sigma'^{-2\alpha})\|_{L^2(\Omega_i)} \le C \|H-H'\|_{W^{1,\infty}(X)}. 
	\label{eq:gradsig}
    \end{align}
    Further, from the pointwise relations $\sigma^{-2\alpha} = |\nabla u_1|^2/H_{11}$ and similarly for $\sigma'$, we write
    \begin{align*}
	\sigma^{-2\alpha} - \sigma'^{-2\alpha} = \frac{1}{H_{11}}\nabla (u_1-u'_1)\cdot\nabla(u_1 + u'_1) + (H_{11}-H'_{11}) \frac{|\nabla u'_{11}|^2}{H_{11}H'_{11}}.
    \end{align*}
    Taking $L^2$ norms and using the triangle inequality, we obtain that 
    \begin{align*}
	\|\sigma^{-2\alpha} - \sigma'^{-2\alpha}\|_{L^2(\Omega_i)} &\le \big( \|\nabla (u_1-u'_1)\|_{L^2} (\|\nabla u_1\|_\infty + \|\nabla u'_1\|_\infty) \\
	&\qquad + \|H_{11} - H'_{11}\|_\infty \|H_{11}^{-1}\|_\infty \|\nabla u'_1\|_{L^2} \|\nabla u'_1\|_\infty \big) \|H_{11}^{-1}\|_\infty,
    \end{align*}
    which again yields an estimate of the form 
    \begin{align}
	\|\sigma^{-2\alpha} - \sigma'^{-2\alpha}\|_{L^2(\Omega_i)} \le C \|H-H'\|_{W^{1,\infty}(X)}. 
	\label{eq:sig}
    \end{align}
    Combining \eqref{eq:gradsig} and \eqref{eq:sig}, we arrive at
    \begin{align*}
	\|\sigma^{-2\alpha} - \sigma'^{-2\alpha}\|_{H^1(\Omega_i)} \le C \|H-H'\|_{W^{1,\infty}(X)},
    \end{align*}
    for every $\Omega_i\in\O$. This concludes the proof. 
\end{proof}

\subsubsection{Discussion on the spectrum of $\bfP_W$}

\paragraph{ \underline{The two-dimensional case}} In the case $n=2$, \cite[Theorem 4]{AN} guarantees that we can pick $m=n$. In so doing and using the form \eqref{eq:ellipticsys3} of the elliptic system together with the fact that $c_F=1$ for all $\alpha\in\Rm$, we arrive at the system
\begin{align*}
    \nabla\cdot( D H^{pi}\nabla u_i) = 0, \quad u_p|_{\partial X} = g_p, \quad p=1,2.
\end{align*}
The weak formulation of the corresponding problem with homogeneous Dirichlet conditions involves the bilinear form 
\begin{align*}
    B (\bfv,\bfv) = \int_\Omega DH^{pi} \nabla v_i\cdot\nabla v_p\ dx.
\end{align*}
Since $H^{-1}$ is uniformly elliptic over $X$ and $\inf_{X} D\ge c_0$, this bilinear form is coercive over $\H$ as seen from the following calculation
\begin{align*}
    \int_\Omega DH^{pi} \nabla v_i\cdot\nabla v_p\ dx = \sum_{k=1}^n \int_\Omega D \langle \partial_k \bfv, H^{-1} \partial_k \bfv \rangle\ dx \ge c_0 \inf_{x\in X} \lambda_M^{-1} \|\bfv\|_{\H}^2,
\end{align*}
where $\lambda_M$ stands for the largest eigenvalue of $H$, for which we have, pointwise ($x_M$ designates a unit eigenvector associated with $\lambda_M$)
\begin{align*}
    \lambda_M = \langle x_M, Hx_M \rangle = \sum_{i,j} H_{ij} x_{M,i} x_{M,j} \le \|H\|_\infty \frac{1}{2} \sum_{i,j} x_{M,i}^2 + x_{M,j}^2 = n\|H\|_\infty,
\end{align*}
and hence the estimate
\begin{align*}
    B (\bfv,\bfv) = \int_\Omega DH^{pi} \nabla v_i\cdot\nabla v_p\ dx \ge c_0 (n\|H\|_\infty)^{-1} \|\bfv\|_\H^2.
\end{align*}
Therefore by virtue of the Lax-Milgram theorem, the system \eqref{eq:ellipticsys} admits a unique solution in $\H$. In particular, this shows that $-c_F^{-1}$ is not an eigenvalue of $\bfP_W$ in this case for any $\alpha\in\Rm$. 

\paragraph{\underline{The case $n\ge 3$}} Using the fact that the spectrum of $\|\bfP_W\|$ is bounded in norm by $\|\bfP_W\|$ and that $\bfP_W$ is compact, the elliptic system admits a unique and stable solution, except for a discrete set of values $-c_F^{-1} \in [-\|\bfP_W\|,\|\bfP_W\|]$ possibly converging to zero. In terms of $\alpha$, this corresponds to almost all values of $\alpha\in\Rm$ except a sequence $\{\alpha_k\}$ taking values in the interval $[\frac{-\|\bfP_W\|-1}{n-2},\frac{\|\bfP_W\|-1}{n-2}]$ and possibly converging to $-(n-2)^{-1}$.

\paragraph{\underline{The special case $\alpha=0$}} In this case we have $c_F=1$. This implies that whenever one can ensure the positivity condition \eqref{eq:positivity} with only $m=n$ solutions (e.g. in even dimension and using lemma \ref{lem:CGO}), one can rewrite the system into the form \eqref{eq:ellipticsys3} with term $1-c_F=0$, that is 
\begin{align*}
    \nabla\cdot(DH^{pi}\nabla u_i) = 0, \quad 1\le p\le n. 
\end{align*}
Using the same arguments as in the two-dimensional case, this system is coercive and therefore ensures that $-1$ is not an eigenvalue of $\bfP_W$.

\paragraph{\underline{Conclusion}}
As a conclusion of this discussion, the following statements hold:
\begin{enumerate}
    \item if $n=2$, we have $\sp \bfP_W \cap \{ -((n-2)\alpha+1), \alpha\in\Rm \} = \emptyset$,
    \item if $n\ge 3$, then $\sp \bfP_W \cap \{ -((n-2)\alpha+1), \alpha\in\Rm\}$ consists of at most a sequence $\{ -((n-2)\alpha_k +1), k=1,2\dots \}$ where $\alpha_k$ belongs to $[\frac{-\|\bfP_W\|-1}{n-2},\frac{\|\bfP_W\|-1}{n-2}]$ and possibly converges to $-(n-2)^{-1}$. In the case $m=n$, the value $0$ is excluded from the latter interval. 
\end{enumerate}

%
\section{Constraints, reconstructions, and compatibility conditions}
\label{sec:comp}

The ODE-based reconstructions use the full redundancy of the data to construct an overdetermined system of equations for the vectors $S_i$ (or $R_i$) and the vector $\nabla\log\sigma$. The PDE-based method defines a well-posed system of equations for the scalar quantities $u_i$ and an overdetermined system of equations for vector $\nabla\log\sigma$. Each of these overdetermined systems needs to satisfy compatibility conditions in order to admit a solution. In this section, we aim to extract information from the over-determinacy of the system. We first revisit the two-dimensional case and use the redundancy to extract explicit reconstruction algorithms in the setting $\alpha\ne\frac{1}{2}$. We then consider the case of arbitrary dimension and show that the compatibility conditions that data must satisfy in order for the aforementioned systems to have solutions take the form of vanishing appropriately defined curvatures together with the cancellation of a given two-form. These conditions generate quadratic functionals of the unknown vectors $S_i$ or $R_i$ whose pratical applicability is discussed below.

\subsection{Reconstructions in two dimensions}

In this section, we revisit the two-dimensional case which was first solved in \cite{BBMT,CFGK} and generalize the approach to the case $\alpha\ne \frac{1}{2}$.  In that approach, the reconstruction of $F = \nabla\log\sigma$ requires the reconstruction of a function $\theta:X\to\Sm^1$ that characterizes the unknown information about the frames $S$ or $R$. We consider the $SO_2(\Rm)$-valued $R= (R_1,R_2)$ frame and parameterize it as $R_1 = (\cos\theta,\sin\theta)^T$ and $R_2 = JR_1$, with $J:= \left[\begin{smallmatrix} 0 & -1 \\ 1 & 0 \end{smallmatrix} \right]$. With the notation $\Phi_{ij} := R_i\otimes R_j$ for $i,j=1,2$, we recast equations \eqref{eq:RtoF} and \eqref{eq:divRi}  for $\alpha\in\Rm$ as follows:
\begin{align}
    \nabla\log\sigma &= F = \nabla\log D + 2\sum_{i,j=1}^2 \Phi_{ij} V_{ij}^s, \label{eq:datatoF2D} \\
    \nabla\cdot R_i &= V_{ik}\cdot R_k - (1-\alpha)F\cdot R_i, \quad i=1,2 \label{eq:divRi2D}.
\end{align}
We next derive an equation for $\nabla\theta$, which by construction is nothing but $[R_2,R_1]$. We have:
\begin{align}
    \nabla\theta = [R_2,R_1] = \del_{R_2} R_1 - \del_{R_1} R_2 = -\Gamma_{22}^1 R_2 + \Gamma_{11}^2 R_1,
    \label{eq:gradtheta0}
\end{align}
where the Christoffel symbols $\Gamma_{11}^2$ and $\Gamma_{22}^1$ are given by 
\begin{align}
    \begin{split}
	\Gamma_{11}^2 &= \del_{R_1}R_1\cdot R_2 = [\nabla, R_1] = -\nabla\cdot R_2 \\
	&= - ( (2\alpha-1)V_{22}-(1-\alpha)N )\cdot R_2 - ( (2\alpha-1)V_{12}^s - V_{12}^a )\cdot R_1, \\
	\Gamma_{22}^1 &= \del_{R_2}R_2\cdot R_1 = -[\nabla, R_2] = -\nabla\cdot R_1 \\
	&= - ( (2\alpha-1)V_{11}-(1-\alpha)N )\cdot R_1 - ( (2\alpha-1)V_{12}^s + V_{12}^a )\cdot R_2.	
    \end{split}
    \label{eq:chris2D}
\end{align}
By orthonormality the other Christoffel symbols are given by 
\begin{align*}
    \Gamma_{11}^1 = \Gamma_{21}^1 = \Gamma_{12}^2 = \Gamma_{22}^2 = 0, \quad \Gamma_{12}^1 = -\Gamma_{11}^2, \qandq \Gamma_{21}^2 = -\Gamma_{22}^1.
\end{align*}
Plugging the expressions \eqref{eq:chris2D} into \eqref{eq:gradtheta0}, we arrive at
\begin{align}
    \nabla\theta = V_{12}^a - (1-\alpha)J\nabla\log D + (2\alpha-1) (\Phi_{21}V_{11} - \Phi_{12}V_{22} + (\Phi_{22}-\Phi_{11})V_{12}^s).
    \label{eq:gradtheta2}
\end{align}
Using the following identity 
\begin{align*}
    \Phi_{21}V_{11} - \Phi_{12}V_{22} + (\Phi_{22}-\Phi_{11})V_{12}^s &= J(\Phi_{11} V_{11} + \Phi_{22} V_{22} + (\Phi_{12} + \Phi_{21})V_{12}^s) \\
    &= \frac{1}{2} J(F-\nabla \log D),  
\end{align*}
equation \eqref{eq:gradtheta2} may be recast as 
\begin{align}
    \nabla\theta = V_{12}^a - \frac{1}{2} J\nabla\log D + \Big( \alpha-\frac{1}{2} \Big) JF,
    \label{eq:gradtheta3}
\end{align}
whose expression matches the one given in \cite{BBMT,CFGK} when $\alpha = \frac{1}{2}$. Since $F$ is a function of $\theta$, the above equation is then a non-linear PDE whenever $\alpha\ne\frac{1}{2}$. This is to be contrasted with the seemingly much nicer case $\alpha =\frac{1}{2}$, whose r.h.s. is independent of $\theta$. 

A right-hand side independent of $\theta$ can, however, be obtained by taking divergence of both sides of \eqref{eq:gradtheta3} since $F= \nabla\log \sigma$ and $\nabla\cdot(J\nabla)=0$. The equation we obtain is 
\begin{align}
    \Delta\theta = \nabla\cdot V_{12}^a.
    \label{eq:laptheta}
\end{align}
This elliptic PDE requires knowledge of $\theta$ at the domain's boundary. Assume that we know $u_i|_{\partial X} = g_i$, $\J_i = \sigma\partial_{\nu} u_i$ for $i=1,2$, and $\sigma$ at the boundary. In this setting, we find that 
\begin{align*}
    \theta|_{\partial X} &= \arg(t_{11} \nabla u_1+t_{12} \nabla u_2|_{\partial X}) \\
    &= \arg ( (t_{11} \partial_t g_1 + t_{12} \partial_t g_2)\ \bft + \sigma^{-1} (t_{11} \J_1+t_{12} \J_2)\ \bmnu),
\end{align*}
with $\bmnu$ and $\bft=J\bmnu$ the unit outgoing normal vector and its direct orthogonal vector, respectively. 

Once $\theta$ is reconstructed, we know the r.h.s. of \eqref{eq:datatoF2D} and solve for $\log\sigma$, either by integrating \eqref{eq:datatoF2D} along a curve, or by taking the divergence of both sides of \eqref{eq:datatoF2D} and solving a Poisson equation provided that $\sigma|_{\partial X}$ is known. Note that the inversion for $\theta$ and $\log\sigma$ by means of the elliptic equations \eqref{eq:laptheta} and ``divergence of \eqref{eq:datatoF2D}'' with Dirichlet conditions is {\em unique} and {\em Lipschitz-stable} in $H^2(X)$ w.r.t. the data $H_{ij}$. The details are left to the reader.

We now discuss the compatibility conditions for the gradient equations \eqref{eq:datatoF2D} and \eqref{eq:gradtheta3}, which admit a solution only if their respective r.h.s. are curl-free. Such conditions lead to a better understanding of the range of the measurement operator and are necessary to ensure that reconstructions based on ODE integrations do not depend on the choice of integration path. 

\subsection{Compatibility conditions in two dimensions of space}

The compatibility conditions for \eqref{eq:gradtheta3} and \eqref{eq:datatoF2D} are that $\nabla\cdot(J\nabla\theta)=0$ and $\nabla\cdot(J\nabla\log\sigma)=0$, respectively. For $\alpha\ne\frac{1}{2}$, these equations not only provide constraints on the redundant data, but in fact give us direct information about the unknown coefficients. In the two-dimensional case, they allow us to solve algebraically for $\cos(2\theta),\sin(2\theta)$, which in turn characterizes $F$  in terms of the data (and therefore does not require the prior resolution of $\theta$). 

Let us first simplify the expression of $F$ as follows:
\begin{align*}
    F &= \nabla\log D + 2(\Phi_{11}V_{11} + \Phi_{22}V_{22} + (\Phi_{12}+\Phi_{21})V_{12}^s) \\
    &= -V_{11} - V_{22} + 2(\Phi_{11}V_{11} + \Phi_{22}V_{22} + (\Phi_{12}+\Phi_{21})V_{12}^s) \\
    &= (\Phi_{11}-\Phi_{22}) (V_{11}-V_{22}) + (\Phi_{12}+\Phi_{21})(V_{12}+V_{21}), 
\end{align*}
where the matrices $\Phi_{11}-\Phi_{22}$ and $\Phi_{12}+\Phi_{21}$ are reflexion matrices that can be expressed in the following manner:
\begin{align}
    \begin{split}
	\Phi_{11}-\Phi_{22} &= c_2\Um + s_2 J\Um \qandq \Phi_{12}+\Phi_{21} = -s_2\Um+c_2 J\Um, \where \\
	(c_2,s_2)&:= (\cos(2\theta),\sin(2\theta)), \quad \Um := \left[
	\begin{array}{cc}
	    1 & 0 \\ 0 & -1
	\end{array}
	\right].
    \end{split}
    \label{eq:reflexions}
\end{align}
As a result, we are able to express $F$ in a rather compact way
\begin{align}
    \begin{split}
	F(x,\theta) &= \cos(2\theta) F_c(x) + \sin(2\theta) JF_c(x), \where \\
	F_c(x) &:= \Um (V_{11}-V_{22}) + J\Um (V_{12}+V_{21}).	
    \end{split}    
    \label{eq:Fnice}
\end{align}
Note the property that $\partial_\theta F = 2JF$. Now turning to the compatibility conditions proper, equations \eqref{eq:gradtheta3} and \eqref{eq:Fnice} are well-defined only if their curls are zero, which gives the following two scalar conditions
\begin{align}
    \nabla\cdot(JG)- \Big( \alpha-\frac{1}{2} \Big) \nabla\cdot (F(x,\theta(x)) &= 0 \qandq \nabla\cdot (JF(x,\theta(x))) = 0, \nonumber \\
    \where\quad G:= V_{12}^a - \frac{1}{2}J\nabla\log D. \label{eq:G}
\end{align}
Now using the chain rule and $\nabla\theta = G + \left(\alpha-\frac{1}{2}\right)JF$, we have
\begin{align*}
    \nabla\cdot(JF) &= c_2 \nabla\cdot (JF_c) - s_2 \nabla\cdot F_c + \nabla\theta\cdot (J\partial_\theta F) \\
    &= c_2 \nabla\cdot (JF_c) - s_2 \nabla\cdot F_c - 2 G \cdot (c_2F_c + s_2 JF_c) = f c_2 - g s_2,
\end{align*}
where we have defined 
\begin{align}
    f(x) := \nabla\cdot(JF_c)-2F_c\cdot G \qandq g(x) := \nabla\cdot F_c + 2JF_c\cdot G.
    \label{eq:fg}
\end{align}
Similarly, the second compatibility equation can be recast as
\begin{align*}
    \Big(\alpha-\frac{1}{2}\Big) (g c_2 + f s_2 ) = \nabla\cdot (JG) - 2\Big( \alpha-\frac{1}{2} \Big)^2 |F_c|^2.
\end{align*}

If $\alpha\ne\frac{1}{2}$, we thus see that the two compatibility equations imply the system 
\begin{align}
    \left[
    \begin{array}{cc}
	f & -g \\ g & f
    \end{array}
    \right] \left[
    \begin{array}{c}
	c_2 \\ s_2
    \end{array}
    \right] = \left[
    \begin{array}{c}
	0 \\ h
    \end{array}
    \right], \quad h:= \Big( \alpha-\frac{1}{2} \Big)^{-1} \nabla\cdot (JG) - 2\Big( \alpha-\frac{1}{2} \Big) |F_c|^2,
    \label{eq:h}
\end{align}
which may be inverted as 
\begin{align*}
    \cos(2\theta) = c_2 = \frac{gh}{f^2+g^2} \qandq \sin(2\theta) = s_2 = \frac{fh}{f^2+g^2}.
\end{align*}
Note that this solution makes sense only if the functions $f,g,h$ are such that $c_2^2 + s_2^2 = 1$, that is, if they satisfy the relation 
\begin{math}
    f^2 + g^2 = h^2. 
\end{math}
In this case, $F$ may be expressed as 
\begin{align*}
    F = \nabla\log\sigma = \frac{h}{f^2+g^2} \left( g F_c + f JF_c  \right).
\end{align*}
The right-hand-side is guaranteed to be curl-free by construction. Inserting \eqref{eq:fg} into the last equation and using $(u\cdot Jv)v - (u\cdot v)Jv = |v|^2 u$ with $u = G$ and $v=F_c$, we obtain the following {\em explicit} reconstruction formula:
\begin{align}
    \nabla\log\sigma = \frac{h}{f^2+g^2} ( (\nabla\cdot F_c) F_c + \nabla\cdot (JF_c) JF_c + 2 |F_c|^2 G ).
    \label{eq:Ffinal}
\end{align}

When $\alpha=\frac{1}{2}$, the equation $\nabla\cdot(J\nabla\theta)=0$ {\em depends solely on the data} and reads
\begin{align*}
    \nabla\cdot(JV_{12}^a) - \frac{1}{2}\Delta\log D = 0.
\end{align*}
The other compatibility equation $\nabla\cdot(JF)=0$ is still of the form $fc_2 + gs_2 = 0$, with $f,g$ defined above, which by itself only gives us $(c_2,s_2)$ up to a sign, i.e. 
\begin{align*}
    (c_2,s_2) = \pm (f^2+g^2)^{-\frac{1}{2}} (-g,f). 
\end{align*}
The above constraint provides partial answer about $\theta$ that may be used in practical reconstructions to mitigate the influence of noise in the data. Reconstructions based solely on these algebraic relations, however, seem to be less stable than  the two approaches based on integration of gradient or Laplace equations.

\subsection{Compatibility conditions in higher dimensions}
\label{sec:compintcond}

The two-dimensional case is special in that $SO_2(\Rm)$ is both one dimensional and Abelian. This is not the case in higher dimensions, where parameterizations are much more complicated, even in three dimensions. Similar compatibility conditions arise in the $n$-dimensional case for gradient equations as an application of the Poincar\'e lemma:  the ``curl'' (or exterior derivative) of the equation vanishes on both sides. For equations \eqref{eq:datatoF} or \eqref{eq:RtoF}, we have that
\begin{align}
    d F^\flat = d^2 \log\sigma = 0,
    \label{eq:dFzero}
\end{align}
which implies a system of $\frac12n(n-2)$ scalar equations (i.e. the cancellation of a 2-form in dimension $n$). 

Regarding the systems \eqref{eq:gradSi} and \eqref{eq:gradRi}, their complete integrability is equivalent to ensuring that the curvature tensor of the Euclidean connection is identically zero when expressed in either frame $S$ or $R$. Indeed, according to \cite[Theorem 1 p30]{caratheodory}, a system of the form 
\begin{align*}
    \partial_k E_i^j(x) = F_{ijk} (E(x),x), \quad 1\le i,j,k\le n, \quad E:= \{E_i^j\}_{i,j=1}^n,
\end{align*}
is integrable if and only if the following conditions hold
\begin{align*}
    \partial_l F_{ijk} + F_{pql} \partial_{E_p^q} F_{ijk} = \partial_k F_{ijl} + F_{pqk} \partial_{E_p^q} F_{ijl},
\end{align*}
which is equivalent, after using the chain rule, to
\begin{align*}
    \partial_l\partial_k E_i^j - \partial_k\partial_l E_i^j = \partial_l (F_{ijk} (E(x),x)) - \partial_k (F_{ijl} (E(x),x)) = 0.
\end{align*}
The last equation is nothing but the fact that the curvature $(\R(\bfe_l,\bfe_k)E_i)\cdot \bfe_j$ of the Euclidean metric is zero for every quadruple $(i,j,k,l)$, where the curvature tensor $\R$ is defined, for three vector fields $X,Y,Z$, by
\begin{align}
    \R(X,Y)Z := \del_X \del_Y Z - \del_Y \del_X Z - \del_{[X,Y]} Z.
    \label{eq:curvature}
\end{align}
$\R$ is known to be a tensor, i.e. it is linear over smooth functions in all its arguments, and thus the above zero curvature conditions are strictly equivalent to the equations
\begin{align*}
    (\R(E_p,E_q)E_k)\cdot E_r = (\R(E_p^i\bfe_i,E_q^j\bfe_j)E_k)\cdot E_r^l\bfe_l = E_p^iE_q^jE_j^l (\R(\bfe_i,\bfe_j)E_k)\cdot\bfe_l = 0, 
\end{align*}
for $1\le p,q,k,r\le n$, where $E$ is any frame, including $S$ and $R$. As one can see from \cite[Prop. 10 and 12 pp 196-197]{Sp2}, this highly redundant set of $n^4$ scalar equations is equivalent to $\frac12n(n-1)$ non-redundant equations which express the cancellation of the {\em sectional curvatures} 
\begin{align}
    (\R(E_i,E_j)E_i)\cdot E_j = 0, \qquad 1\le i< j\le n.
    \label{eq:zerocurvatureE}
\end{align}

For each frame, we thus have the two systems of $\frac12n(n-1)$ equations \eqref{eq:dFzero} and \eqref{eq:zerocurvatureE}. We now work with the $R$ frame because its Christoffel symbols have nicer symmetry properties and show that both systems \eqref{eq:dFzero} and \eqref{eq:zerocurvatureE} may be recast as 
\begin{align}
    \sum_{p,q=1}^n \Mm_{ij}^{pq}:R_p\otimes R_q = 0, \quad 1\le i<j\le n,
    \label{eq:contractions}
\end{align}
where the matrices $\Mm_{ij}^{pq}$ depend only on the data (we could obtain similar equations for the $S$ frame). We first rewrite the Christoffel symbols \eqref{eq:chrisR} under the form
\begin{align}
    \begin{split}
	\Gamma_{ij}^k &= W_{ij}^{kl}\cdot R_l, \where \\
	W_{ij}^{kl} &:= 2\alpha c_F (\delta_{ik} V_{lj}^s - \delta_{ij} V_{lk}^s) + \delta_{li} V_{jk}^a + \delta_{lj} \left( V_{ik}^s + \alpha c_F N \delta_{ik} \right) - \delta_{lk} \left( V_{ij}^s + \alpha c_F N \delta_{ij}  \right).
    \end{split}
    \label{eq:chrisW}    
\end{align} 
The vector fields $W_{ij}^{kl}$ depend only on the data and have the antisymmetry properties $W_{ij}^{kl} = - W_{ik}^{jl}$ and $W_{ij}^{jl} = 0$. We now derive systems of the form \eqref{eq:contractions} for both systems \eqref{eq:dFzero} and \eqref{eq:zerocurvatureE} in the case of the $R$ frame. 

\subsubsection{The condition $d F^\flat=0$.}
Starting from equation \eqref{eq:RtoF}, the equation $d F^\flat=0$ reads
\begin{align*}
    0 = \frac{1}{2} c_F^{-1} d F^\flat &= \frac{1}{2} d^2\log D + d( ( V_{kl}^s \cdot R_k)  R^\flat_l   ) = d( (V_{kl}^s \cdot R_k)  R^\flat_l).
\end{align*}
Now using identity \eqref{eq:dcarac} with vector fields $R_i, R_j\ (i\ne j)$, we have
\begin{align*}
    d( (V_{kl}^s \cdot R_k)  R^\flat_l) (R_i,R_j) &= \del_{R_i}( V_{kj}^s\cdot R_k) - \del_{R_j} (V_{ki}^s\cdot R_k) - (V_{kl}^s\cdot R_k)R_l\cdot[R_i,R_j] \\
    &= \del_{R_i}( V_{kj}^s\cdot R_k) - \del_{R_j}(V_{ki}^s\cdot R_k) \\
    &\qquad + (V_{kl}^s\cdot R_k) ( (W_{ij}^{lp}-W_{ji}^{lp})\cdot R_p) \,\,=\,\,0.
\end{align*}
Decomposing the first term in the last r.h.s. as follows
\begin{align*}
    \del_{R_i}( V_{kj}^s\cdot R_k) &= \del_{R_i} V_{kj}^s \cdot R_k + V_{kj}^s\cdot\del_{R_i} R_k = R_k\cdot \del_{R_i} V_{kj}^s + (V_{kj}^s\cdot R_p)(W_{ik}^{pq}\cdot R_q), 
\end{align*}
and doing simlarly for the second term, we obtain the set of scalar equations
\begin{align}
    \begin{split}
	R_k\cdot (\del_{R_i} V_{kj}^s &- \del_{R_j} V_{ki}^s) + (V_{kj}^s\cdot R_p)(W_{ik}^{pq}\cdot R_q) - (V_{ki}^s\cdot R_p)(W_{jk}^{pq}\cdot R_q) \\
	&\quad + (V_{ql}^s\cdot R_q) ( (W_{ij}^{lp}-W_{ji}^{lp})\cdot R_p) = 0, \quad 1\le i<j\le n.
    \end{split}    
    \label{eq:dFzeroR}
\end{align}
This system can be written in the form \eqref{eq:contractions}, where the matrices $\Mm_{pq}^{ij}$ depend only on the data. 

\subsubsection{The zero curvature conditions.}
Given the symmetries of the Christoffel symbols, one can show for an orthonormal frame that the zero sectional curvature equations \eqref{eq:zerocurvatureE} can be recast as
\begin{align}
    \del_{R_i}\Gamma_{jj}^i + \del_{R_j}\Gamma_{ii}^j = - \Gamma_{ji}^l \Gamma_{ij}^l + \Gamma_{ii}^l \Gamma_{jj}^l - (\Gamma_{ij}^l - \Gamma_{ji}^l)\Gamma_{li}^j, \quad 1\le i<j\le n.
    \label{eq:zerocurvR2}
\end{align}
Using expression \eqref{eq:chrisW} of the Christoffel symbols, the first term in the left-hand side of \eqref{eq:zerocurvR2} may be rewritten as 
\begin{align*}
    \del_{R_i}\Gamma_{jj}^i = \del_{R_i} W_{jj}^{il}\cdot R_l + W_{jj}^{il} \cdot\del_{R_i} R_l &= R_l\cdot\del_{R_i} W_{jj}^{il} + \Gamma_{il}^k W_{jj}^{il}\cdot R_k \\
    &= R_l\cdot\del_{R_i} W_{jj}^{il} + (W_{il}^{qp}\cdot R_p)(W_{jj}^{il}\cdot R_q).
\end{align*}
Proceeding similarly for the second term of the l.h.s. of \eqref{eq:zerocurvR2} and plugging expression \eqref{eq:chrisW} into the r.h.s., we obtain the set of equations for $1\le i<j\le n$
\begin{align}
    \begin{split}
	R_l\cdot (\del_{R_i} W_{jj}^{il} &+ \del_{R_j} W_{ii}^{jl}) + (W_{il}^{qp}\cdot R_p)(W_{jj}^{il}\cdot R_q) + (W_{jl}^{qp}\cdot R_p)(W_{ii}^{jl}\cdot R_q) \\
	&= -(W_{ij}^{lp}\cdot R_p)(W_{ji}^{lq}\cdot R_q) + (W_{ii}^{lp}\cdot R_p)(W_{jj}^{lq}\cdot R_q) \\
	&\qquad + ( (W_{ij}^{lp}-W_{ji}^{lp})\cdot R_p)(W_{li}^{jq}\cdot R_q)	
    \end{split}
    \label{eq:zerocurvR3}
\end{align}
($i,j$ are not being summed over but $l,p,q$ are). This is also a quadratic system of the form \eqref{eq:contractions} with different matrices $\Mm_{pq}^{ij}$.

\subsubsection{Discussion}

Based on the result of the two-dimensional case, we make the following heuristic statements: depending on the value of $\alpha$, these compatibility equations 
\begin{itemize}
    \item[(i)] either give us compatibility conditions on the data (that do not depend on the unknown frame), thus characterizing the range of the measurement operator,
    \item[(ii)] or they may allow us to invert algebraically for the cosines and sines of the $n(n-1)/2$ spherical angles that parameterize the $SO_n(\Rm)$-valued $R$ frame. This in turn may come at the price of other compatibility conditions that only depend on the data.
\end{itemize}
It remains an interesting, so far unresolved, question to find an algorithm that enforces the compatibility conditions as the system of ODEs is used to ensure that the reconstruction does not depend on the choice of integration paths.

\subsection{Remark on the elliptic method}
The system of elliptic equations for the scalar solutions $u_i$ is well-posed for almost all values of $\alpha$ and $n$. Once the solutions $u_i$ are obtained, it remains to solve the equation for $F=\nabla \log\sigma$. The only remaining compatibility condition is therefore that the latter term indeed be a gradient.  In a similar manner to what we just saw for the ODE-based method, writing the condition $d F^\flat=0$ yields $\frac12 n(n-1)$ equations of the type 
\begin{align*}
    \sum_{p,q} \Mm_{pq}^{ij}:\nabla u_p\otimes \nabla u_q = 0, \quad 1\le i<j\le n,
\end{align*}
where the matrices $\Mm_{pq}^{ij}$ depend only on the data $H_{ij}$.




\medskip
Received xxxx 20xx; revised xxxx 20xx.
\medskip


\begin{thebibliography}{99}
	\bibitem{ABCTF} 
	    \newblock H. Ammari, E. Bonnetier, Y. Capdeboscq, M. Tanter and M. Fink,
	    \newblock {\em Electrical impedance tomography by elastic deformation,}
	    \newblock SIAM J. Appl. Math., 68({\bf 6}) (2008), pp. 1557--1573.
	\bibitem{AN} 
	    \newblock G. Alessandrini and V. Nesi,
	    \newblock {\em Univalent $e^\sigma$-harmonic mappings,}
	    \newblock Arch. Rat. Mech. Anal., {\bf 158} (2001), pp. 155--171.
	\bibitem{B} 
	    \newblock G. Bal,
	    \newblock {\em Hybrid inverse problems and internal functionals (review paper),}
	    \newblock in ``Inside Out'' (ed. Gunther Uhlmann), Cambridge University Press (2012). 
	\bibitem{B-UMEIT-11} 
	    \newblock \leavevmode\vrule height 2pt depth -1.6pt width 23pt,
	    \newblock {\em {Cauchy problem and Ultrasound modulated EIT},}
	    \newblock submitted.
	\bibitem{BBMT} 
	    \newblock G. Bal, E. Bonnetier, F. Monard and F. Triki,
	    \newblock {\em Inverse diffusion from knowledge of power densities}, 
	    \newblock Inverse Probl. Imaging, in press (2012).
	\bibitem{BR-IP-11} 
	    \newblock G. Bal and K. Ren,
	    \newblock {\em Multi-source quantitative photoacoustic tomography},
	    \newblock Inverse Problems, {\bf 27} (2011).
	\bibitem{BU-IP-10} 
	    \newblock G. Bal and G. Uhlmann,
	    \newblock {\em Inverse diffusion theory for photoacoustics},
	    \newblock Inverse Problems, 26({\bf 8}) (2010), p.~085010.
	\bibitem{calderon80} 
	    \newblock A. Calder{\'o}n,
	    \newblock {\em On an inverse boundary value problem},
	    \newblock Seminar on Numerical Analysis and its Applications to Continuum Physics, Soc. Brasileira de Matematica, Rio de Janeiro,  (1980), pp.~65--73.
	\bibitem{caratheodory} 
	    \newblock C. Carath\'eodory,
	    \newblock ``Calculus of Variations and Partial Differential Equations of the First Order'',
	    \newblock 3ed., AMS Chelsea, 1999.
	\bibitem{CFGK} 
	    \newblock Y. Capdeboscq, J. Fehrenbach, F. de Gournay and O. Kavian,
	    \newblock {\em Imaging by Modification: Numerical Reconstruction of Local Conductivities from Corresponding Power Density Measurements},
	    \newblock Siam Journal on Imaging Sciences, {\bf 2} (2009), pp. 1003--1030. 
	\bibitem{E} 
	    \newblock L.C. Evans,
	    \newblock ``Partial Differential Equations'',
	    \newblock Graduate Studies in Mathematics, Vol. 19, AMS (1998).
	\bibitem{GS-SIAP-09}
	    \newblock B. Gebauer and O. Scherzer,
	    \newblock {\em Impedance-acoustic tomography},
	    \newblock SIAM J.  Applied Math.,  69({\bf 2}) (2009), pp.~565--576.
	\bibitem{KKSY-SIMA-02} 
	    \newblock S. Kim, O. Kwon, J. K. Seo, and J.-R. Yoon,
	    \newblock {\em On a Nonlinear Partial Differential Equation Arising in Magnetic Resonance Electrical Impedance Tomography},
	    \newblock SIAM J. Math. Anal., {\bf 34} (2002), pp.~511--526.
	\bibitem{KK} 
	    \newblock P. Kuchment and L. Kunyansky,
	    \newblock {\em 2D and 3D reconstructions in acousto-electric tomography},
	    \newblock Inverse Problems {\bf 27} (2011), 055013.
	\bibitem{L} 
	    \newblock J.M. Lee,
	    \newblock ``Riemannian Manifolds, An Introduction to Curvature'',
	    \newblock Graduate Texts in Mathematics, Vol. 176, Springer (1997).
	\bibitem{M} 
	    \newblock F. Monard,
	    \newblock ``Taming unstable inverse problems. Mathematical routes toward high-resolution medical imaging modalities'',
	    \newblock Ph.D. thesis, Columbia University in the city of New York (2012).
	\bibitem{NTT-IP-07} 
	    \newblock A. Nachman, A. Tamasan, and A. Timonov,
	    \newblock {\em Conductivity imaging with a single measurement of boundary and interior data},
	    \newblock Inverse Problems, {\bf 23} (2007), pp.~2551--2563.
	\bibitem{NTT-IP-09} 
	    \newblock \leavevmode\vrule height 2pt depth -1.6pt width 23pt,
	    \newblock {\em Recovering the conductivity from a single measurement of interior data},
	    \newblock Inverse Problems, {\bf 25} (2009), p.~035014.
	\bibitem{NTT-Rev-11} 
	    \newblock \leavevmode\vrule height 2pt depth -1.6pt width 23pt,
	    \newblock {\em Current density impedance imaging},
	    \newblock Contemporary Mathematics, American Mathematical Society, in press (2012).
	\bibitem{S-Sp-2011} 
	    \newblock O. Scherzer,
	    \newblock ``Handbook of Mathematical Methods in Imaging'',
	    \newblock Springer Verlag, New York (2011).
	\bibitem{Sp2} 
	    \newblock M. Spivak,
	    \newblock ``A comprehensive introduction to Differential Geometry, Vol. 2'',
	    \newblock 2nd Ed., Publish or perish (1990).
	\bibitem{Syl-Uhl-87}
	    \newblock J. Sylvester and G. Uhlmann,
	    \newblock {\em A global uniqueness theorem for an inverse boundary value problem},
	    \newblock Ann. of Math., 125({\bf 1}) (1987), pp.~153--169. 
	\bibitem{T} 
	    \newblock M. Taylor,
	    \newblock ``Partial Differential Equations I, Basic Theory'',
	    \newblock Springer New York (1996).
\end{thebibliography}
\end{document}